\documentclass[english,review]{elsarticle}
\usepackage[T1]{fontenc}
\usepackage[latin9]{inputenc}
\usepackage{float}
\usepackage{mathtools}
\usepackage{url}
\usepackage{amsmath}
\usepackage{amsthm}
\usepackage{amssymb}
\usepackage{graphicx}

\makeatletter

\newcommand{\lyxdot}{.}

\floatstyle{ruled}
\newfloat{algorithm}{tbp}{loa}
\providecommand{\algorithmname}{Algorithm}
\floatname{algorithm}{\protect\algorithmname}

\theoremstyle{plain}
\newtheorem{thm}{\protect\theoremname}
\theoremstyle{plain}
\newtheorem{lem}[thm]{\protect\lemmaname}
\ifx\proof\undefined
\newenvironment{proof}[1][\protect\proofname]{\par
	\normalfont\topsep6\p@\@plus6\p@\relax
	\trivlist
	\itemindent\parindent
	\item[\hskip\labelsep\scshape #1]\ignorespaces
}{%
	\endtrivlist\@endpefalse
}
\providecommand{\proofname}{Proof}
\fi
\theoremstyle{remark}
\newtheorem{rem}[thm]{\protect\remarkname}
\theoremstyle{plain}
\newtheorem{cor}[thm]{\protect\corollaryname}

\usepackage[colorlinks = true,
            linkcolor = blue,
            urlcolor  = blue,
            citecolor = blue,
            anchorcolor = blue]{hyperref}
\usepackage{algpseudocode}
\usepackage{times}

\@ifundefined{showcaptionsetup}{}{%
 \PassOptionsToPackage{caption=false}{subfig}}
\usepackage{subfig}
\makeatother

\usepackage{babel}
\providecommand{\corollaryname}{Corollary}
\providecommand{\lemmaname}{Lemma}
\providecommand{\remarkname}{Remark}
\providecommand{\theoremname}{Theorem}

\begin{document}
\begin{frontmatter}

\title{ARPIST: Provably Accurate and Stable \\
Numerical Integration over Spherical Triangles}

\author[1,2]{Yipeng Li} \ead{jamesonli1313@gmail.com}
\author[1]{Xiangmin Jiao\corref{cor1}} \ead{xiangmin.jiao@stonybrook.edu}
\cortext[cor1]{Corresponding author} \address[1]{Dept. of Applied Math. \& Stat. and Institute for Advanced Computational Science, Stony Brook University, Stony Brook, NY 11794, USA.} \address[2]{Current address: Beijing Oneflow Technology Ltd., Haidian District, Beijing, 100083, China} 
\begin{abstract}
Numerical integration on spheres, including the computation of the
areas of spherical triangles, is a core computation in geomathematics.
The commonly used techniques sometimes suffer from instabilities and
significant loss of accuracy. We describe a new algorithm, called
\emph{ARPIST}, for accurate and stable integration of functions on
spherical triangles. ARPIST is based on an easy-to-implement transformation
to the spherical triangle from its corresponding linear triangle via
radial projection to achieve high accuracy and efficiency. More importantly,
ARPIST overcomes potential instabilities in computing the Jacobian
of the transformation, even for poorly shaped triangles that may occur
at poles in regular longitude-latitude meshes, by avoiding potential
catastrophic rounding errors. We compare our proposed technique with
L'Huilier's Theorem for computing the area of spherical triangles,
and also compare it with the recently developed LSQST method (J. Beckmann,
H.N. Mhaskar, and J. Prestin, \emph{GEM - Int. J. Geomath.}, 5:143--162,
2014) and a radial-basis-function-based technique (J. A. Reeger and
B. Fornberg, \emph{Stud. Appl. Math}., 137:174--188, 2015) for integration
of smooth functions on spherical triangulations. Our results show
that ARPIST enables superior accuracy and stability over previous
methods while being orders of magnitude faster and significantly easier
to implement.
\end{abstract}
\end{frontmatter}

\section{Introduction}

Applications in geophysics often require solving partial differential
equations (PDEs) on spherical geometries using numerical methods,
such as finite element and finite volume methods. A critical component
is the computation of numerical integration over the elements or cells
of a surface mesh discretizing a sphere, for example, to compute surface
fluxes between the atmosphere and ocean models. Sometimes, it is desirable
to compute the integration accurately to near machine precision, for
example, when transferring flux quantities under the constraint of
global conservation. In these applications, one often needs to integrate
over a spherical $n$-gon. In this work, we focus on integrating functions
over spherical triangles, of which the edges are geodesics between
the vertices since a spherical $n$-gon can be tessellated into $n-2$
triangles.

Despite its importance, accurate and stable computation on spherical
triangles has not been resolved satisfactorily in the literature.
Classical numerical quadrature (a.k.a. cubature) techniques on spheres
have focused on integrating smooth functions over a whole sphere,
for example, by determining a minimal number of quadrature points
to maximize the exact integration of a maximal number of spherical
harmonics or polynomials; see, e.g., the classical papers \citep{mclaren1963optimal,lebedev1976quadratures},
recent monographs \citep{atkinson2012spherical,dai2013approximation},
the comprehensive survey article \citep{hesse2015numerical} and the
references therein, as well as more recent works \citep{reeger2016numerical,portelenelle2018efficient}.
Some of these techniques compute the integration using a tessellation
of the sphere, such as a spherical Delaunay triangulation \citep{reeger2016numerical}
or a cubed-sphere mesh \citep{portelenelle2018efficient}. One of
the earliest works that focused on integration over spherical triangles
is \citep{atkinson1982numerical}, which integrated over a spherical
triangle by recursively subdividing it and then using low-degree quadrature
rules. Such a technique, however, may require too many splittings
to reach (near) machine precision. Beckmann et al. \citep{beckmann2014local}
proposed the so-called LSQST, which uses QR factorization to compute
quadrature weights on each spherical triangle, resulting in superlinear
complexity in the number of quadrature points. In addition, LSQST
suffers from numerical instabilities \citep[Remark 2.3]{beckmann2014local}.
Recently, A. Sommariva et al. \citep{sommariva2021near,sommariva2021numerical}
used perpendicular projection  and quadrature rules on elliptical
sectors for spherical integration, which can reach machine precision
with a high degree quadrature rule (e.g., $n=20$) on a sphere octant.
Another recent work, SphericalQuadratureRBF (or SQRBF in short) \citep{reeger2015sqrbf,reeger2016numerical},
utilizes radial-basis functions in computing integration on a spherical
triangulation.

In this work, we propose a new technique called \emph{anchored radially
projected integration on spherical triangles} or \emph{ARPIST}. The
core idea of ARPIST is to utilize the transformation to the spherical
triangle from its corresponding linear triangle via a radial projection.
Although this idea is simple, care must be taken to ensure the stable
computation of the Jacobian determinant of the transformation. We
prove that ARPIST can overcome potential instabilities, even for poorly
shaped triangles, such as those near poles in a regular latitude-longitude
(RLL) mesh or the triangles in overlay meshes in remapping algorithms
\citep{ullrich2016arbitrary}. We overcome the instabilities by properly
selecting one of the vertices as the ``anchor'' when computing the
Jacobian determinant to avoid catastrophic cancellation errors. As
a result, ARPIST achieves superior accuracy and stability, reaching
(near) machine precision with an adaptive-refinement procedure. ARPIST
is also highly efficient: It requires only linear time complexity
in the number of quadrature points per triangle, which is significantly
more efficient than LSQST. In addition, ARPIST is much easier to implement
than LSQST \citep{beckmann2014local} and SQRBF \citep{reeger2016numerical}.
The MATLAB and Python implementations of ARPIST are available at \url{https://github.com/numgeom/arpist}.

The remainder of the paper is organized as follows. In Section~\ref{sec:Accurate-and-Stable},
we derive a new algorithm for numerical integration on spherical triangles
and analyze its accuracy and stability. In Section~\ref{sec:Numerical-Experiments},
we present some comparisons of ARPIST with other techniques in the
literature for the computation of the areas of spherical triangles
and the integration of smooth functions. Section~\ref{sec:Conclusions}
concludes the paper with some discussions.

\section{\label{sec:Accurate-and-Stable}Accurate and Stable Integration on
Spheres}

We describe a new algorithm to compute the numerical integration of
a sufficiently smooth function over a spherical triangle. 

\subsection{Integration via radial projection}

Consider a spherical triangle $S$ with vertices $\boldsymbol{x}_{1}$,
$\boldsymbol{x}_{2}$, and $\boldsymbol{x}_{3}$. Without loss of
generality, assume the vertices are in counterclockwise order with
respect to the outward normal to the sphere. Let $r$ denote the radius
of the sphere, i.e., $r=\Vert\boldsymbol{x}_{i}\Vert$. Let $T$ denote
its corresponding flat (linear) triangle $\boldsymbol{x}_{1}\boldsymbol{x}_{2}\boldsymbol{x}_{3}$,
and let $(\xi,\eta)$ denote the natural coordinates of $T$, so that
$T$ has the parameterization
\begin{equation}
\boldsymbol{x}(\xi,\eta)=(1-\xi-\eta)\boldsymbol{x}_{1}+\xi\boldsymbol{x}_{2}+\eta\boldsymbol{x}_{3}\label{eq:r(xi,eta)}
\end{equation}
for $0\leq\xi\leq1$ and $0\leq\eta\leq1-\xi$. Let $\hat{\boldsymbol{p}}(\boldsymbol{x})=r\boldsymbol{x}/\left\Vert \boldsymbol{x}\right\Vert $,
which projects a point $\boldsymbol{x}\in T$ onto a point in $S$.
We then obtain a radial projection 
\begin{equation}
\boldsymbol{p}(\xi,\eta)=\hat{\boldsymbol{p}}(\boldsymbol{x}(\xi,\eta))=r\boldsymbol{x}(\xi,\eta)/\left\Vert \boldsymbol{x}(\xi,\eta)\right\Vert .\label{eq:points_sphere}
\end{equation}
See Figure~\ref{fig:Projection-from} for a schematic of the mapping.
\begin{figure}
\begin{centering}
\includegraphics[width=0.98\columnwidth]{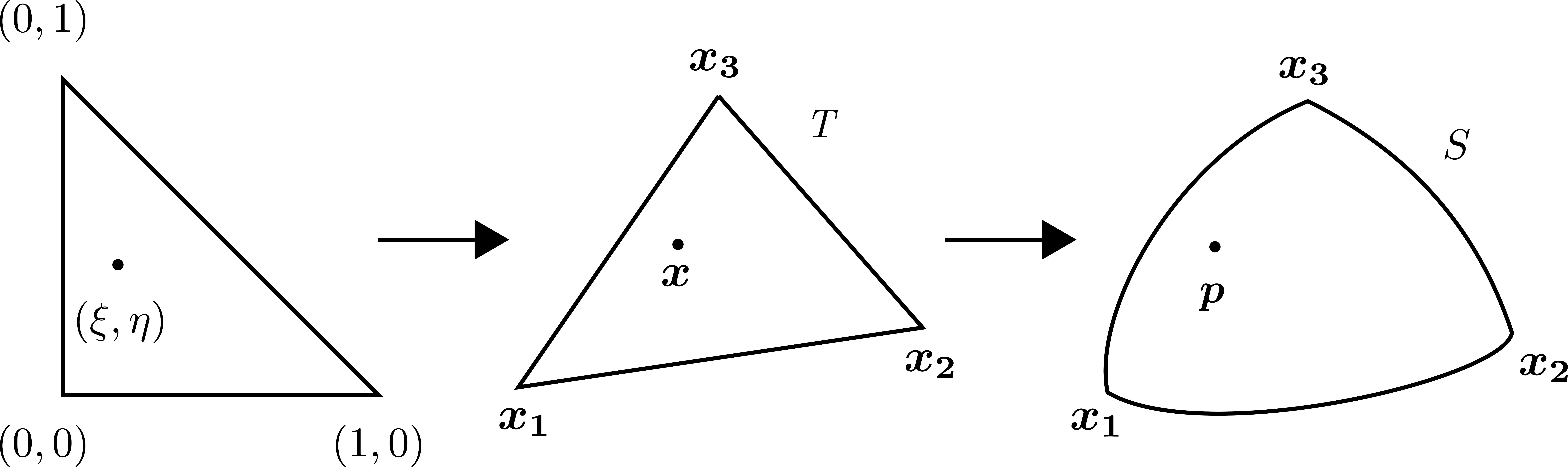}
\par\end{centering}
\caption{\label{fig:Projection-from}Projection from a reference triangle to
a spherical triangle via a linear triangle.}
\end{figure}

Given a function $f$ on the spherical triangle $S$, the integral
is
\begin{equation}
\int_{S}f(\boldsymbol{p})dA=\int_{0}^{1}\int_{0}^{1-\xi}f(\boldsymbol{p}(\xi,\eta))J\,\text{d}\eta\,\text{d}\xi,\label{eq:integral-transformed}
\end{equation}
where $J=\left\Vert \boldsymbol{p}_{\xi}\times\boldsymbol{p}_{\eta}\right\Vert $
is the Jacobian determinant of the mapping from the reference triangle
to the curved triangle. The partial derivatives $\boldsymbol{p}_{\xi}$
and $\boldsymbol{p}_{\eta}$ have the closed forms
\begin{align}
\boldsymbol{p}_{\xi} & =\frac{r}{\left\Vert \boldsymbol{x}\right\Vert }\left(\boldsymbol{x}_{\xi}-\frac{\boldsymbol{x}\cdot\boldsymbol{x}_{\xi}}{\boldsymbol{x}\cdot\boldsymbol{x}}\boldsymbol{x}\right),\label{eq:dp_dxi}\\
\boldsymbol{p}_{\eta} & =\frac{r}{\left\Vert \boldsymbol{x}\right\Vert }\left(\boldsymbol{x}_{\eta}-\frac{\boldsymbol{x}\cdot\boldsymbol{x}_{\eta}}{\boldsymbol{x}\cdot\boldsymbol{x}}\boldsymbol{x}\right).\label{eq:dp_deta}
\end{align}
Note that $J$ varies from point to point, so the computation based
on (\ref{eq:dp_dxi}) and (\ref{eq:dp_deta}) directly are inefficient
to compute. More importantly, the cross product may suffer from instabilities
due to cancellation errors when $\boldsymbol{p}_{\xi}$ and $\boldsymbol{p}_{\eta}$
are nearly parallel.

To achieve stability and efficiency, we derive a new formula as follows.
For brevity, let us use the notation
\[
\det\left[\boldsymbol{a},\boldsymbol{b},\boldsymbol{c}\right]=\left|\begin{array}{ccc}
a_{1} & b_{1} & c_{1}\\
a_{2} & b_{2} & c_{2}\\
a_{3} & b_{3} & c_{3}
\end{array}\right|=\left|\begin{array}{ccc}
a_{1} & a_{2} & a_{3}\\
b_{1} & b_{2} & b_{3}\\
c_{1} & c_{2} & c_{3}
\end{array}\right|,
\]
which is equivalent to the triple product $\boldsymbol{a}\cdot(\boldsymbol{b}\times\boldsymbol{c})$.
The following property will turn out to be useful:
\begin{equation}
\det\left[\boldsymbol{a},\boldsymbol{b},\boldsymbol{c}\right]=\det\left[\boldsymbol{a},\boldsymbol{b}+s\boldsymbol{a},\boldsymbol{c}+t\boldsymbol{a}\right]\qquad\forall s,t\in\mathbb{R}.\label{eq:mix product property}
\end{equation}

\begin{lem}
\label{lem:radial-projection}Given a spherical triangle $S$ on a
sphere with radius $r$, let $\boldsymbol{x}_{1}$, $\boldsymbol{x}_{2}$,
and $\boldsymbol{x}_{3}$ be its vertices in counterclockwise order
w.r.t. the outward normal to $S$. Let $T$ be the corresponding linear
triangle $\boldsymbol{x}_{1}\boldsymbol{x}_{2}\boldsymbol{x}_{3}$
composed of points $\boldsymbol{x}(\xi,\eta)$ as in (\ref{eq:r(xi,eta)}).
The integral of a continuous function $f$ over $S$ is 
\begin{equation}
\int_{S}f(\boldsymbol{p})\,\text{d}A=r^{2}\det\left[\boldsymbol{x}_{1},\boldsymbol{x}_{2},\boldsymbol{x}_{3}\right]\int_{0}^{1}\int_{0}^{1-\xi}\frac{f(\boldsymbol{p}(\xi,\eta))}{\left\Vert \boldsymbol{x}(\xi,\eta)\right\Vert ^{3}}\,\text{d}\eta\,\text{d}\xi,\label{eq:radial-projection}
\end{equation}
where $\boldsymbol{p}(\xi,\eta)$ is defined in (\ref{eq:points_sphere}).
\end{lem}
\begin{proof}
Consider $\boldsymbol{p}_{\xi}$ and $\boldsymbol{p}_{\eta}$ in (\ref{eq:dp_dxi})
and (\ref{eq:dp_deta}). Note that $\boldsymbol{p}_{\xi}$ and $\boldsymbol{p}_{\eta}$
are tangent to $S$, so $\boldsymbol{p}_{\xi}\times\boldsymbol{p}_{\eta}$
is normal to $S$ and hence is parallel to $\boldsymbol{x}/\Vert\boldsymbol{x}\Vert$.
The area measure $J$ in (\ref{eq:integral-transformed}) is then
\begin{align}
J=\left\Vert \boldsymbol{p}_{\xi}\times\boldsymbol{p}_{\xi}\right\Vert  & =\text{det\ensuremath{\left[\frac{\boldsymbol{x}}{\left\Vert \boldsymbol{x}\right\Vert },\boldsymbol{p}_{\xi},\boldsymbol{p}_{\eta}\right]}}\label{eq:area_measure}\\
 & =\frac{r^{2}}{\left\Vert \boldsymbol{x}\right\Vert ^{3}}\det\left[\boldsymbol{x},\boldsymbol{x}_{\xi}-\frac{\boldsymbol{x}\cdot\boldsymbol{x}_{\xi}}{\boldsymbol{x}\cdot\boldsymbol{x}}\boldsymbol{x},\boldsymbol{x}_{\eta}-\frac{\boldsymbol{x}\cdot\boldsymbol{x}_{\eta}}{\boldsymbol{x}\cdot\boldsymbol{x}}\boldsymbol{x}\right],\\
 & =\frac{r^{2}}{\left\Vert \boldsymbol{x}\right\Vert ^{3}}\text{det\ensuremath{\left[\boldsymbol{x},\boldsymbol{x}_{\xi},\boldsymbol{x}_{\eta}\right]}},\label{eq:area_measure_2}
\end{align}
where the last equality is due to (\ref{eq:mix product property}).
From (\ref{eq:r(xi,eta)}), we have 
\begin{equation}
\boldsymbol{x}_{\xi}=\boldsymbol{x}_{2}-\boldsymbol{x}_{1}\qquad\text{and}\qquad\boldsymbol{x}_{\eta}=\boldsymbol{x}_{3}-\boldsymbol{x}_{1}.\label{eq:dr_dxi}
\end{equation}
Substituting (\ref{eq:dr_dxi}) and (\ref{eq:r(xi,eta)}) into (\ref{eq:area_measure_2})
and using (\ref{eq:mix product property}), we have
\begin{align}
J & =\frac{r^{2}}{\left\Vert \boldsymbol{x}\right\Vert ^{3}}\text{det\ensuremath{\left[\boldsymbol{x},\boldsymbol{x}_{2}-\boldsymbol{x}_{1},\boldsymbol{x}_{3}-\boldsymbol{x}_{1}\right]}}\label{eq:area_measure_3}\\
 & =\frac{r^{2}}{\left\Vert \boldsymbol{x}\right\Vert ^{3}}\text{det\ensuremath{\left[\boldsymbol{x}_{1}+\xi(\boldsymbol{x}_{2}-\boldsymbol{x}_{1})+\eta(\boldsymbol{x}_{3}-\boldsymbol{x}_{1}),\boldsymbol{x}_{2}-\boldsymbol{x}_{1},\boldsymbol{x}_{3}-\boldsymbol{x}_{1}\right]}}\\
 & =\frac{r^{2}}{\left\Vert \boldsymbol{x}\right\Vert ^{3}}\text{det\ensuremath{\left[\boldsymbol{x}_{1},\boldsymbol{x}_{2}-\boldsymbol{x}_{1},\boldsymbol{x}_{3}-\boldsymbol{x}_{1}\right]}}\label{eq:area_measure_4}\\
 & =\frac{r^{2}}{\left\Vert \boldsymbol{x}\right\Vert ^{3}}\text{det\ensuremath{\left[\boldsymbol{x}_{1},\boldsymbol{x}_{2},\boldsymbol{x}_{3}\right]}}.\label{eq:area_measure_final}
\end{align}
Substituting (\ref{eq:area_measure_final}) into (\ref{eq:integral-transformed}),
we then obtain (\ref{eq:radial-projection}).
\end{proof}
Note that in (\ref{eq:area_measure}), the triple product $\det\left[\boldsymbol{x}/\left\Vert \boldsymbol{x}\right\Vert ,\boldsymbol{p}_{\xi},\boldsymbol{p}_{\eta}\right]$
is guaranteed to be positive due to the counterclockwise convention
of $\boldsymbol{x}_{1}\boldsymbol{x}_{2}\boldsymbol{x}_{3}$, since
both $\boldsymbol{x}$ and $\boldsymbol{p}_{\xi}\times\boldsymbol{p}_{\eta}$
points outward to $S$. 
\begin{rem}
Eq.~(\ref{eq:radial-projection}) has some vague similarities as
(2) in \citep{reeger2016numerical}, but there are two fundamental
differences. First, Reeger and Fornberg used the gnomonic projection
\citep{snyder1987map} to project spherical triangles onto a local
tangent plane after rotating the spherical triangle to the top. In
contrast, our approach projects the spherical triangles onto the linear
triangle radially without requiring any rotation. Second, Reeger and
Fornberg expressed the integration in the global $xyz$ coordinate
system and then approximated it using radial-basis functions. In contrast,
our approach expresses the integration using a local parameterization
over the triangle and then integrates the function using standard
Gaussian quadrature rules over the linear triangles.
\end{rem}

\subsection{Radially projected Gaussian quadrature}

Utilizing Lemma~\ref{lem:radial-projection}, we then obtain an efficient
quadrature rule based on the Gaussian quadrature over the linear triangle
$T$.
\begin{thm}
\label{thm:rp-quadrature}Given a spherical triangle $S$ on a sphere
with radius $r$ and  vertices and its corresponding linear triangle
$T$ as in Lemma~\ref{lem:radial-projection}, let $\{(\boldsymbol{\xi}_{i},w_{i})\mid1\leq i\leq q\}$
define a degree-$p$ quadrature over $T$, where the $\boldsymbol{\xi}_{i}$
are the quadrature points and the $w_{i}$ are the corresponding weights.
If the integrand $f(\boldsymbol{p}):S\rightarrow\mathbb{R}$ is continuously
differentiable to $p$th order, then
\begin{equation}
\int_{S}f(\boldsymbol{p}(\boldsymbol{\xi}))dA=r^{2}\det\left[\boldsymbol{x}_{1},\boldsymbol{x}_{2},\boldsymbol{x}_{3}\right]\sum_{i}\frac{w_{i}}{\left\Vert \boldsymbol{x}(\boldsymbol{\xi}_{i})\right\Vert ^{3}}f(\boldsymbol{p}(\boldsymbol{\xi}_{i}))+A\mathcal{O}(h^{p+1}),\label{eq:rp_quadrature}
\end{equation}
where $h$ denotes the longest edge of the triangle and $A=\text{area}(T)$.
\end{thm}
This theorem follows from the two-dimensional Taylor series expansion
\citep{humpherys2017foundations} and the high-order chain rule \citep{ma2009higher},
analogous to the proof of one-dimensional quadrature rules. We sketch
the argument as follows. 
\begin{proof}
Let $g(\boldsymbol{\xi})$ denote $f(\boldsymbol{p}(\boldsymbol{\xi}))/\left\Vert \boldsymbol{x}(\boldsymbol{\xi})\right\Vert ^{3}$.
Consider the $d$-dimensional Taylor series expansion of $g(\boldsymbol{\xi})$
with respect to $\boldsymbol{\xi}$ about $\boldsymbol{\xi}_{0}$
(for example, $\boldsymbol{\xi}_{0}=\boldsymbol{0}$),
\begin{align}
g(\boldsymbol{\xi}_{0}+\boldsymbol{h}) & =\sum_{k=0}^{p}\frac{1}{k!}\boldsymbol{\nabla}^{k}g(\boldsymbol{\xi}_{0}):\boldsymbol{h}^{k}+\frac{C_{1}\left\Vert \boldsymbol{\nabla}^{p+1}g(\boldsymbol{\xi}_{0}+\boldsymbol{\epsilon})\right\Vert }{(p+1)!}\Vert\boldsymbol{h}\Vert^{p+1},\label{eq:Taylor-polynomial-residual}
\end{align}
where $\boldsymbol{h}^{k}$ denotes the $k$th tensor power of $\boldsymbol{h}$,
$\boldsymbol{\nabla}^{k}$ denotes the derivative tensor of order
$k$ with respect to $\boldsymbol{\xi}$, ``:'' denotes the scalar
product of $k$th-order tensors, $\Vert\boldsymbol{\epsilon}\Vert\leq\Vert\boldsymbol{h}\Vert$,
and $\vert C_{1}\vert\leq1$. A degree-$p$ quadrature rule over the
triangle $T$ integrates the first term in (\ref{eq:Taylor-polynomial-residual})
exactly. Since $\boldsymbol{x}(\boldsymbol{\xi})$ and $\boldsymbol{p}(\boldsymbol{\xi})$
are both smooth and $\left\Vert \boldsymbol{x}(\boldsymbol{\xi})\right\Vert =r+\mathcal{O}(h^{2})\geq C_{2}>0$,
by repeatedly applying the high-order chain rule, we conclude that
$\left\Vert \boldsymbol{\nabla}^{p+1}g(\boldsymbol{\xi}_{0}+\boldsymbol{\epsilon})\right\Vert $
is bounded if $f(\boldsymbol{p})$ is continuously differentiable.
Hence, the remainder term in (\ref{eq:Taylor-polynomial-residual})
is $\mathcal{O}(h^{p+1})$, and the integral error over $T$ is bounded
by $\text{area}(T)\mathcal{O}(h^{p+1})$.
\end{proof}
Theorem~\ref{thm:rp-quadrature} allows us to reuse Gaussian quadrature
rules on linear triangles, such as those in \citep{cools2003encyclopedia},
to integrate over spherical triangles directly. It is well known that
Gaussian quadrature rules can achieve the highest accuracy for a given
number of quadrature points over individual triangles, because they
are constructed by solving for the quadrature points and the weights
to maximize the degree of polynomials that can be integrated exactly
(see e.g., \citep[Chapter 8]{heath2018scientific}). In (\ref{eq:rp_quadrature}),
$\text{area}(T)=\mathcal{O}(h^{2})$ for a triangle $T$ with maximum
edge length $h$. When applying the radially projected quadrature
rule on a triangulation of a fixed area on a sphere, the resulting
composite quadrature rule is then $(p+1)$st order accurate if the
integrand is continuously differentiable to $p$th order. 

\subsection{\label{subsec:Stable-determinant}Stable computation of the determinant}

When evaluating (\ref{eq:rp_quadrature}), a subtle numerical issue
is the computation of $\text{det\ensuremath{\left[\boldsymbol{x}_{1},\boldsymbol{x}_{2},\boldsymbol{x}_{3}\right]}}$.
There are various ways in computing the determinant, and they may
have drastically different stability properties when using floating-point
arithmetic for small spherical triangles. For example, one could compute
the determinants using the triple-product formula
\begin{align}
\boldsymbol{a}\cdot(\boldsymbol{b}\times\boldsymbol{c})= & a_{1}(b_{2}c_{3}-b_{3}c_{2})+a_{2}(b_{3}c_{1}-b_{1}c_{3})+a_{3}(b_{1}c_{2}-b_{2}c_{1}),\label{eq:triple-product}
\end{align}
by substituting $\boldsymbol{x}_{1}$, $\boldsymbol{x}_{2}$, and
$\boldsymbol{x}_{3}$ as $\boldsymbol{a}$, $\boldsymbol{b}$, and
$\boldsymbol{c}$, respectively. This approach, however, is unstable
when $\boldsymbol{x}_{1}$, $\boldsymbol{x}_{2},$ and $\boldsymbol{x}_{3}$
are nearly parallel to each other (c.f. Remark~\ref{rem:instability-triple-product}
in \ref{sec:Anchored-triple-product}), which unfortunately is the
case for nearly every triangle on a finer triangulation. Another standard
approach is to use LU factorization (a.k.a. Gaussian elimination)
with partial pivoting (LUPP) \citep[p. 114]{Golub13MC}. In particular,
given $\boldsymbol{A}=\left[\boldsymbol{x}_{1},\boldsymbol{x}_{2},\boldsymbol{x}_{3}\right]$,
one can compute its LUPP 
\begin{equation}
\boldsymbol{P}\boldsymbol{A}=\boldsymbol{L}\boldsymbol{U},\label{eq:lupp}
\end{equation}
where $\boldsymbol{P}$ is a permutation matrix corresponding to the
row interchanges of $\boldsymbol{A}$. Since $\det(\boldsymbol{P}^{T})=\det(\boldsymbol{P})=\pm1$
and $\det(\boldsymbol{L})=1$, 
\begin{equation}
\det(\boldsymbol{A})=\pm\det(\boldsymbol{U})=\pm\prod_{i=1}^{3}u_{ii},\label{eq:det_lupp}
\end{equation}
where the sign is the same as $\text{sign}(\det(\boldsymbol{P}))$
and the $u_{ii}$ are the diagonal entries of $\boldsymbol{U}$. This
technique is more stable than the triple product, but it is still
inaccurate when $\boldsymbol{x}_{1}$, $\boldsymbol{x}_{2},$ and
$\boldsymbol{x}_{3}$ are nearly parallel to each other due to the
large condition number of $\boldsymbol{A}$. A more sophisticated
technique was proposed by Clarkson \citep{clarkson1992safe}, who
adapted modified Gram-Schmidt and used an adaptive procedure to bound
relative errors. However, Clarkson's algorithm applies only to integer
matrices. To the best of our knowledge, there was no existing method
with guaranteed accuracy of the computed determinant with floating-point
arithmetic, even for $3\times3$ matrices.

To achieve accuracy and stability, we notice that
\begin{equation}
\det\left[\boldsymbol{x}_{1},\boldsymbol{x}_{2},\boldsymbol{x}_{3}\right]=\det\left[\boldsymbol{x}_{k},\boldsymbol{x}_{k\oplus1}-\boldsymbol{x}_{k},\boldsymbol{x}_{k\oplus2}-\boldsymbol{x}_{k}\right],\label{eq:det_anchored}
\end{equation}
where $\boldsymbol{x}_{k}$ is chosen to be the vertex incident on
the two shorter edges of the triangle $\boldsymbol{x}_{1}\boldsymbol{x}_{2}\boldsymbol{x}_{3}$,
and $k\oplus i$ denotes $\mod\left(k+(i-1),3\right)+1$. We refer
to $\boldsymbol{x}_{k}$ as the \emph{anchor}. Note that when $k=1$,
(\ref{eq:det_anchored}) coincides with the triple-product in (\ref{eq:area_measure_4}).
When $k\neq1$, shifting the indices as in (\ref{eq:det_anchored})
preserves the determinant. To evaluate the right-hand side of (\ref{eq:det_anchored}),
we propose to use (\ref{eq:triple-product}), which we refer to as
\emph{anchored triple product} (or \emph{ATP}), which turns out to
be stable in practice. A more sophisticated strategy is to compute
(\ref{eq:det_anchored}) using (\ref{eq:lupp}--\ref{eq:det_lupp})
with column equilibration \citep[p. 139]{Golub13MC}. In particular,
let $\boldsymbol{A}=\left[\boldsymbol{x}_{k},\boldsymbol{x}_{k\oplus1}-\boldsymbol{x}_{k},\boldsymbol{x}_{k\oplus2}-\boldsymbol{x}_{k}\right]\boldsymbol{D}^{-1}$
with $\boldsymbol{D}=\text{diag}\left\{ \Vert\boldsymbol{x}_{k}\Vert,\Vert\boldsymbol{x}_{k\oplus1}-\boldsymbol{x}_{k}\Vert,\Vert\boldsymbol{x}_{k\oplus2}-\boldsymbol{x}_{k}\Vert\right\} $,
and then 
\begin{equation}
\det\left[\boldsymbol{x}_{1},\boldsymbol{x}_{2},\boldsymbol{x}_{3}\right]=\left|\prod_{i=1}^{3}d_{i}u_{ii}\right|.\label{eq:det_alupp_eq}
\end{equation}
We refer to the latter approach as \emph{anchored LUPP with equilibration}
(or \emph{ALUPPE}), and it is even more stable than ATP for some pathological
cases.

To justify ATP and ALUPPE, let us first derive error bounds for them
for an arbitrary $\boldsymbol{A}\in\mathbb{R}^{3\times3}$. Let $\epsilon_{\text{machine}}$
denote the machine precision, such that $\left|\text{fl}(x)-x\right|\leq\epsilon_{\text{machine}}\vert x\vert$
and $\left|\text{fl}(x)\circledast\text{fl}(y)\right|\leq\epsilon_{\text{machine}}\vert\text{fl}(x)*\text{fl}(y)\vert$
for any real value $x$ and $y$ and any basic floating-point $\circledast$
corresponding to a basic arithmetic operator $*$ (such as addition
and multiplication), barring overflow and underflow. Let $\odot$
and $\otimes$ denote the inner-product and cross-product operators
under floating-point operations.
\begin{thm}
\label{thm:forward-error-LUPP}Given $\boldsymbol{A}$, let $\tilde{\boldsymbol{P}}^{T}\tilde{\boldsymbol{L}}\tilde{\boldsymbol{U}}$
be the LUPP of $\boldsymbol{A}\boldsymbol{D}^{-1}$ using floating-point
arithmetic, where $\boldsymbol{D}=\text{diag}\{\Vert\boldsymbol{a}_{1}\Vert,\Vert\boldsymbol{a}_{2}\Vert,\Vert\boldsymbol{a}_{3}\Vert\}$.
The absolute error of the computed determinant is bounded in the sense
that
\begin{equation}
\left|\prod_{i=1}^{3}d_{i}\tilde{u}_{ii}-\prod_{i=1}^{3}d_{i}u_{ii}\right|=\det(\boldsymbol{D})\mathcal{O}(\epsilon_{\text{machine}}).\label{eq:absolute-error-bound}
\end{equation}
\end{thm}

\begin{thm}
\label{thm:forward-error-TP}Given $\boldsymbol{A}$ composed of columns
$\boldsymbol{a}_{1}$, $\boldsymbol{a}_{2}$, and $\boldsymbol{a}_{3}$,
assuming that $\boldsymbol{a}_{2}$ and $\boldsymbol{a}_{3}$ are
in counterclockwise order with respect to the direction $\boldsymbol{a}_{1}$,
the absolute error of the computed triple product is bounded in the
sense that
\begin{equation}
\left|\boldsymbol{a}_{1}\cdot(\boldsymbol{a}_{2}\times\boldsymbol{a}_{3})-\boldsymbol{a}_{1}\odot(\boldsymbol{a}_{2}\otimes\boldsymbol{a}_{3})\right|=\frac{\text{\ensuremath{\det}}(\boldsymbol{A})}{\sigma}\mathcal{O}(\epsilon_{\text{machine}}),\label{eq:absolute-error-bound-1}
\end{equation}
assuming $\sigma\ll\epsilon_{\text{machine}}$, where $\sigma=\left|\boldsymbol{a}_{1}\cdot(\boldsymbol{a}_{2}\times\boldsymbol{a}_{3})\right|/\prod_{i}\Vert\boldsymbol{a}_{i}\Vert$.
\end{thm}
The proofs of Theorems~\ref{thm:forward-error-LUPP} and \ref{thm:forward-error-TP}
involve some detailed backward error analysis. For completeness, we
present the proofs in \ref{sec:Error-analysis-of-ALUPP} and \ref{sec:Anchored-triple-product},
respectively. 
\begin{rem}
\label{rem:instability-triple-product} At a high level, $1/\sigma$
in Theorem~\ref{thm:forward-error-TP} plays the role of the condition
number in computing the cross product. The assumption of $\sigma\gg\epsilon_{\text{machine}}$
in the theorem is for the ease of presentation, because the triple
product is nonlinear in $\boldsymbol{A}$, so some simplification
is required. If $\sigma$ is close to $\epsilon_{\text{machine}}$,
rounding errors would always dominate. Hence, this assumption does
not lead to a loss of generality from a practical point of view. As
a practical guideline, an accurate and stable algorithm should make
$\sigma$ as large as possible when computing the triple product for
a given triangulation.
\end{rem}
In the context of computing the Jacobian determinant, we need to substitute
$\boldsymbol{x}_{k}$, $\boldsymbol{x}_{k\oplus1}-\boldsymbol{x}_{k}$,
and $\boldsymbol{x}_{k\oplus2}-\boldsymbol{x}_{k}$ for $\boldsymbol{a}_{1}$,
$\boldsymbol{a}_{2}$, and $\boldsymbol{a}_{3}$ in the preceding
theorems. We can then conclude that $\det(\boldsymbol{A})$ is proportional
to $r\text{area}(\boldsymbol{x}_{1}\boldsymbol{x}_{2}\boldsymbol{x}_{3})$,
so is $\det(\boldsymbol{D})$, assuming that the minimum angle of
the triangle $\boldsymbol{x}_{1}\boldsymbol{x}_{2}\boldsymbol{x}_{3}$
is bounded away from 0. Hence, the relative error of the computed
determinant from LUPP is then expected to be $\mathcal{O}(\epsilon_{\text{machine}})$.
For ATP, let $\theta$ denote the angle between $\boldsymbol{x}_{k\oplus1}-\boldsymbol{x}_{k}$
and $\boldsymbol{x}_{k\oplus2}-\boldsymbol{x}_{k}$. Then, $\sigma$
is proportional to $\sin(\theta)$ for sufficiently small $\theta$,
and the assumption $\sigma\gg\epsilon_{\text{machine}}$ is satisfied
when $\theta$ is bounded away from $\epsilon_{\text{machine}}$.
By choosing the anchor $\boldsymbol{x}_{k}$ to be the vertex incident
on the two shorter edges of the triangle $\boldsymbol{x}_{1}\boldsymbol{x}_{2}\boldsymbol{x}_{3}$,
we ensure $\theta$ is the maximum angle in the triangle and hence
is as close to $90^{\circ}$ as possible, so $\det(\boldsymbol{A})/\det(\boldsymbol{D})$
and $\sigma$ are approximately maximized for ALUPPE and ATP, respectively,
and they are accurate and stable for almost all practical applications.
In contrast, for the naive computation of the triple product in (\ref{eq:triple-product}),
$\sigma$ is as small as to $\sin(\theta_{\min})h^{2}$, where $\theta_{\min}$
is the smallest angle in $\boldsymbol{x}_{1}\boldsymbol{x}_{2}\boldsymbol{x}_{3}$.
Hence, a naive computation of the triple product is unstable for small
$h$ and small $\theta_{\min}$, and its error is expected to be at
least $\mathcal{O}(1/h^{2})$ larger than that of ATP.

One shortcoming of the analysis is that it omitted the potential cancellation
errors in $\boldsymbol{x}_{k\oplus1}-\boldsymbol{x}_{k}$ and $\boldsymbol{x}_{k\oplus2}-\boldsymbol{x}_{k}$.
If the spherical triangles are excessively small, these cancellation
errors may dominate, and the relative errors from ALUPPE and ATP could
be arbitrarily large. Fortunately, the cancellation errors are typically
negligible compared to the other errors in practice, even for the
finest meshes.

To demonstrate the validity of our analysis, we compare the standard
triple product, standard LUPP, and anchored computations for small
spherical triangles, along with LUPP with a different vertex as the
anchor, which we refer to as ``off-anchored.'' We generated 1000
random spherical triangles on a unit sphere, of which the longest
edge is $\sim0.01$ and the shortest edge length is $\sim0.0001$.
We applied the algorithms using double-precision floating-point arithmetic
and used variable-precision arithmetic with 32-digits of precision
to compute the reference solutions. As can be seen from Figure~\ref{fig:Comparison-of-different-determinants},
the anchored computations are about six and four orders of magnitude
more accurate than the standard triple product and LUPP, respectively.
The standard triple product is the least accurate due to its potential
reduction of $\sigma$ by a factor of $h^{2}$ compared to ATP. With
our choice of anchor, the ATP and ALUPPE have comparable performance.
The off-anchored TP is two orders of magnitude worse than ATP due
to the reduction of $\sigma$ by a factor of up to $\sin(\theta_{\min})/\sin(\theta_{\max})$,
and similarly for the off-anchored LUPP. ATP has slightly smaller
errors than ALUPPE, probably because ATP involves fewer floating-point
operators and hence less accumulation of rounding errors. Hence, we
use ATP for its better accuracy, efficiency, and simplicity. Note
that the mean and minimum errors of ATP are nearly coincident for
all cases, indicating that the maximum errors are outliers likely
due to near degeneracies (i.e., extremely small triangles). Hence,
we expect the errors to be close to machine precision for spherical
triangles from practical applications, as we will demonstrate in Section~\ref{sec:Numerical-Experiments}.

\begin{figure}
\begin{centering}
\includegraphics[width=0.6\columnwidth]{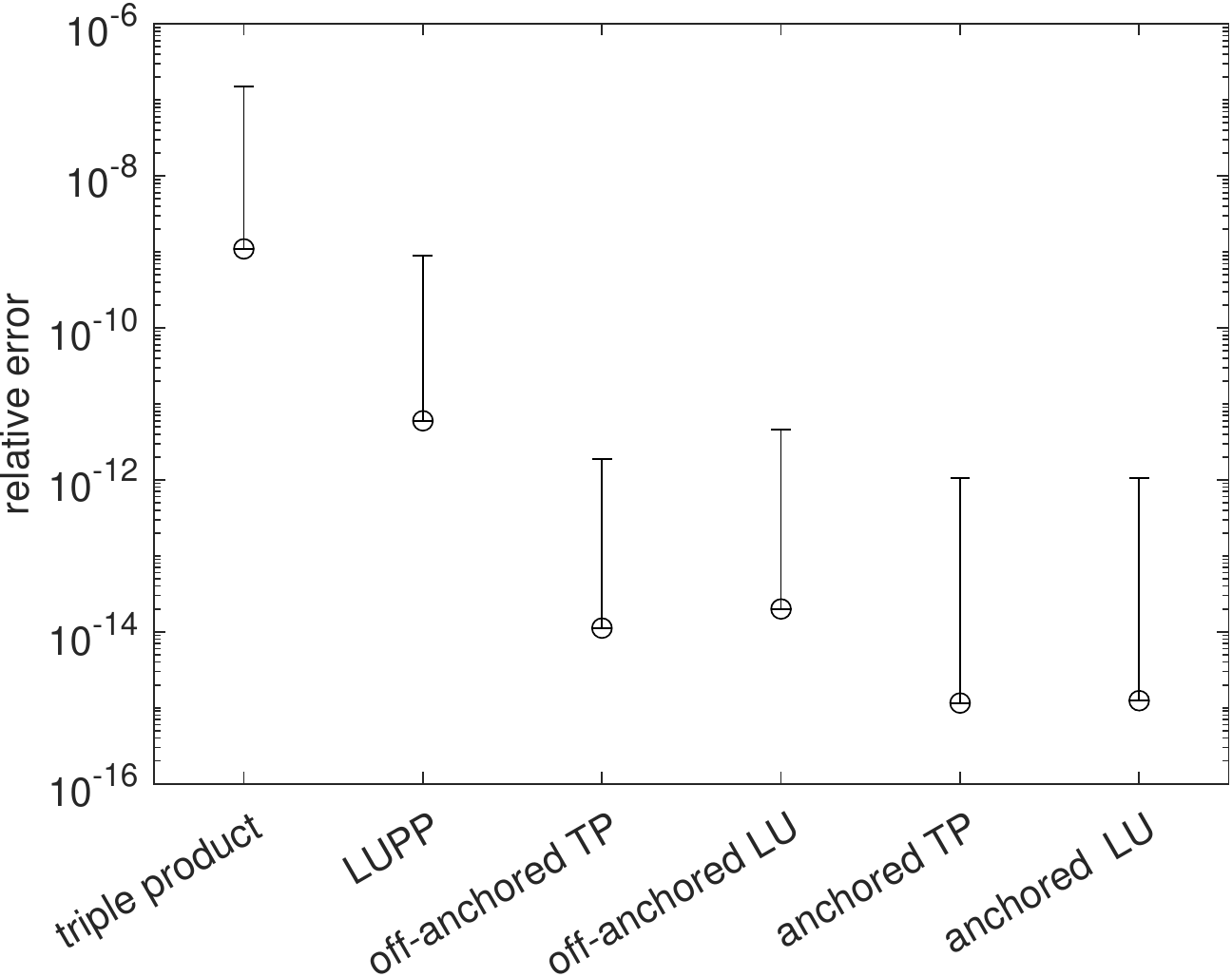}
\par\end{centering}
\caption{\label{fig:Comparison-of-different-determinants}Comparison of different
techniques in computing the determinant of random spherical triangles.
Circles indicate the mean errors, and the vertical bars indicate the
range of the relative errors.}
\end{figure}

\subsection{Anchored radially projected integration on spheres}

We put together the preceding components to obtain an accurate and
stable algorithm for numerical integration on a spherical triangle.
We refer to the algorithm as \emph{anchored radially projected integration
on spherical triangles}, or \emph{ARPIST}. For completeness, Algorithm~\ref{alg:ARPIST}
outlines ARPIST with a given function $f$. We assumed that the quadrature
points $\{\boldsymbol{\xi}_{i}\}$ and the associated weights $\{w_{i}\}$
for the reference triangles are pre-tabulated in the procedure, and
those rules can be found, for example, in \citep{cools2003encyclopedia}
or in the ARPIST GitHub repository. One could replace the function
$f$ by an array of its values at the radially projected quadrature
points, i.e., $\boldsymbol{f}=[f(\boldsymbol{p}(\boldsymbol{\xi}_{i}))]_{i}$.

\begin{algorithm}
\caption{\label{alg:ARPIST}Anchored radially projected integration on a spherical
triangle}

\begin{algorithmic}[1]

\Procedure{arpist}{$\boldsymbol{x}_{1},\boldsymbol{x}_{2},\boldsymbol{x}_{3},f$}\Comment{Integrate
$f$ on spherical tri. $\boldsymbol{x}_{1}\boldsymbol{x}_{2}\boldsymbol{x}_{3}$}

\State $r\leftarrow\Vert\boldsymbol{x}_{1}\Vert;t\leftarrow0$ 

\For{ $i=1$ to npoints}\Comment{Numerical quadrature over flat
tri. $\boldsymbol{x}_{1}\boldsymbol{x}_{2}\boldsymbol{x}_{3}$}

    \State $\boldsymbol{x}\leftarrow\boldsymbol{x}_{1}+\xi_{i}(\boldsymbol{x}_{2}-\boldsymbol{x}_{1})+\eta_{i}(\boldsymbol{x}_{3}-\boldsymbol{x}_{1})$

    \State $t\leftarrow t+w_{i}\,f(r\boldsymbol{x}/\Vert\boldsymbol{x}\Vert)/\Vert\boldsymbol{x}\Vert^{3}$

\EndFor

\State $k\leftarrow\arg\min_{k}\{\Vert\boldsymbol{x}_{k\%3+1}-\boldsymbol{x}_{k}\Vert+\Vert\boldsymbol{x}_{(k+1)\%3+1}-\boldsymbol{x}_{k}\Vert\}$\Comment{Select
anchor}

\State\textbf{return }$tr^{2}\boldsymbol{x}_{k}\cdot\left((\boldsymbol{x}_{k\%3+1}-\boldsymbol{x}_{k})\times(\boldsymbol{x}_{(k+1)\%3+1}-\boldsymbol{x}_{k})\right)$\Comment{Scale
by $r^{2}\det[\boldsymbol{X}]$}

\EndProcedure

\end{algorithmic}
\end{algorithm}

By considering both the truncation and rounding errors, ARPIST can
integrate a sufficiently smooth function stably and accurately, as
stated by the following corollary.
\begin{cor}
\label{cor:overall-error} Given a spherical triangle $S$ with vertices
$\boldsymbol{x}_{1}$, $\boldsymbol{x}_{2}$, and $\boldsymbol{x}_{3}$
in counterclockwise order with its corresponding linear triangle $T$,
if the integrand $f(\boldsymbol{p}):S\rightarrow\mathbb{R}$ is continuously
differentiable to $p$th order, assuming that $\text{det}(\boldsymbol{X})\geq\sigma\Vert\boldsymbol{x}_{1}\Vert\Vert\boldsymbol{x}_{2}\Vert\Vert\boldsymbol{x}_{3}\Vert$
for constant $\sigma\gg\epsilon_{\text{machine}}$, then ARPIST evaluates
the spherical integration with a total error of $\text{area}(T)\mathcal{O}(h^{p+1})+\det(\boldsymbol{X})/\sigma\mathcal{O}(\epsilon_{\text{machine}}).$
\end{cor}
Corollary~\ref{cor:overall-error} directly follows from Theorems~\ref{thm:rp-quadrature}
and \ref{thm:forward-error-LUPP}. Typically, the truncation errors
would dominate, and we expect the integration to approach (near) machine
precision for sufficiently smooth functions. 
\begin{rem}
\label{rem:The-stability-of}The stability of ARPIST makes it well
suited to develop more advanced integration techniques analogous to
their counterparts in 1D. For example, we can develop adaptive quadrature
rules by recursively subdividing the triangles recursively until the
truncation errors are close to machine precision \citep[Section 8.3.6]{heath2018scientific},
as we will demonstrate in Section~\ref{sec:Numerical-Experiments}.
As another example, we can apply Romberg integration by leveraging
Richardson extrapolation to accelerate convergence \citep[Section 8.7]{heath2018scientific}.
\end{rem}

\section{\label{sec:Numerical-Experiments}Numerical Experiments}

In this section, we report numerical experimentation with ARPIST and
compare it with a commonly used technique for computing the areas
of spherical triangles \citep{beyer1991crc} and two recently proposed
techniques for computing spherical integration \citep{beckmann2014local,reeger2016numerical}.

\subsection{\label{subsec:Area-computation-of}Accurate and stable computation
of spherical-triangle area}

We first apply ARPIST to the accurate and stable computation of the
area of a spherical triangle, which is mathematically equivalent to
the integration of unity, i.e., $f(\boldsymbol{p})\equiv1$. This
problem is of particular importance in enforcing global conservation
in earth modeling. Hence, it is desirable to be computed as accurately
as possible and ideally to (near) machine precision.  Presently,
this area computation is typically carried out by applying the centuries-old
theorems due to Girard and L'Huilier, but such a technique is often
observed to be inaccurate.  In particular, due to Girard's theorem,
the area of a spherical triangle $S$ with radius $R$ is mathematically
equal to $A=R^{2}E$, where $E$ is the \emph{spherical excess} of
$S$; due to L'Huilier's Theorem \citep{beyer1991crc},
\begin{equation}
E=4\arctan\sqrt{\tan\left(\frac{s}{2}\right)\tan\left(\frac{s-a_{1}}{2}\right)\tan\left(\frac{s-a_{2}}{2}\right)\tan\left(\frac{s-a_{3}}{2}\right)},\label{eq:LHuilier}
\end{equation}
where the $a_{i}$ are the length of sides on the spherical triangle
and $s=(a_{1}+a_{2}+a_{3})/2$ is the semiperimeter of $S$. Since
(\ref{eq:LHuilier}) is the core of this computation, we refer to
the approach as L'Huilier's theorem or \emph{LT}. Due to its popularity,
we will use LT as the baseline in assessing ARPIST for this problem.
As a side product, we will reveal the numerical instabilities in LT
that have led to the inaccuracy of this popular technique.

For the area computation to be accurate and stable, it needs to be
insensitive to the sizes and shapes of the triangles. More precisely,
it should be stable when the maximum edge length $h$ or the minimum
angle $\theta_{\min}$ tends to $0$ (or tends to $\epsilon_{\text{machine}}^{\alpha}$
for some $0.5\lesssim\alpha<1$ so that rounding errors would not
dominate truncation errors). We assess the accuracy of ARPIST using
double-precision arithmetic as $h$ or $\theta_{\min}$ tends to zero
in Figure~\ref{fig:max-edge-length}. We computed the reference solution
using (\ref{eq:LHuilier}) with 128-digits quadruple-precision floating-point
numbers. For ARPIST, we report the results using degree-4 and degree-8
quadrature rules. In addition, we report the results for an adaptive
procedure as we alluded to in Remark~\ref{rem:The-stability-of}.
The adaptive ARPIST applies degree-4 and degree-8 Gaussian quadrature
rules if $h\leq h_{1}$ and $h_{1}<h\leq h_{2}$, respectively, and
recursively splits a larger triangle if $h>h_{2}$, where $h_{1}$
and $h_{2}$ are determined experimentally. Figure~\ref{fig:max-edge-length}(a)
shows the relative errors in computed areas of spherical triangles
on a unit sphere for $\pi/500\leq\theta_{\min}\leq\pi/3$, where $h\approx0.26$.
It can be seen that degree-4, degree-8, and adaptive ARPIST are all
insensitive to $\theta_{\min}.$ Figure~\ref{fig:max-edge-length}(b)
shows the errors for equilateral triangles with $10^{-3}\leq h\leq1$.
It is clear that degree-4 and degree-8 ARPIST achieved near machine
precision (below $10^{-15}$) for $h\lesssim0.004$ and $h\lesssim0.05$,
respectively. Hence, we set $h_{1}=0.004$ and $h_{2}=0.05$ in adaptive
ARPIST.  

\begin{figure}
\subfloat[$h\approx0.26$ and various $\theta_{\min}$.]{\includegraphics[width=0.5\textwidth]{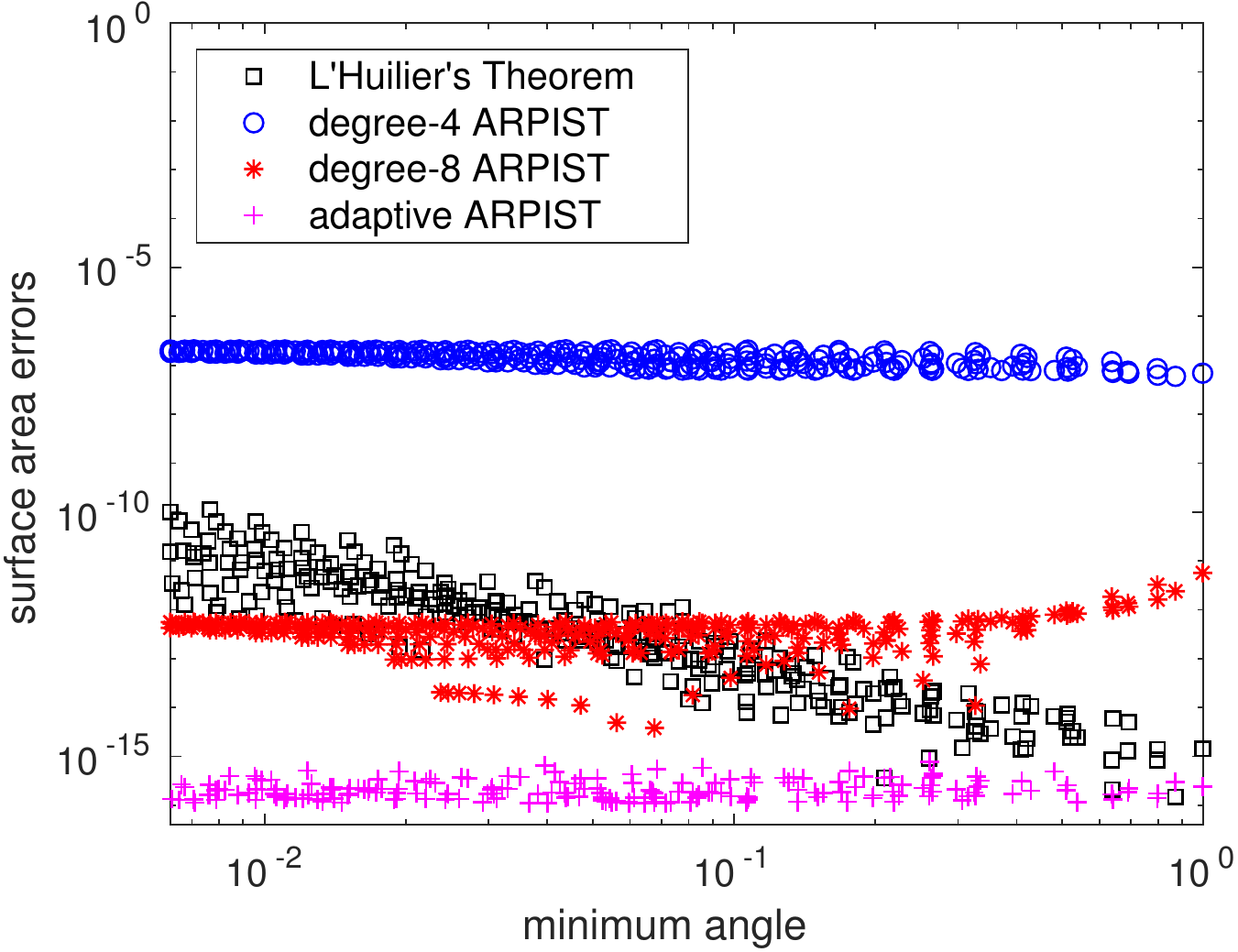}}~\subfloat[Various $h$ and $\theta_{\min}=\text{60}^{\circ}$.]{\includegraphics[width=0.5\textwidth]{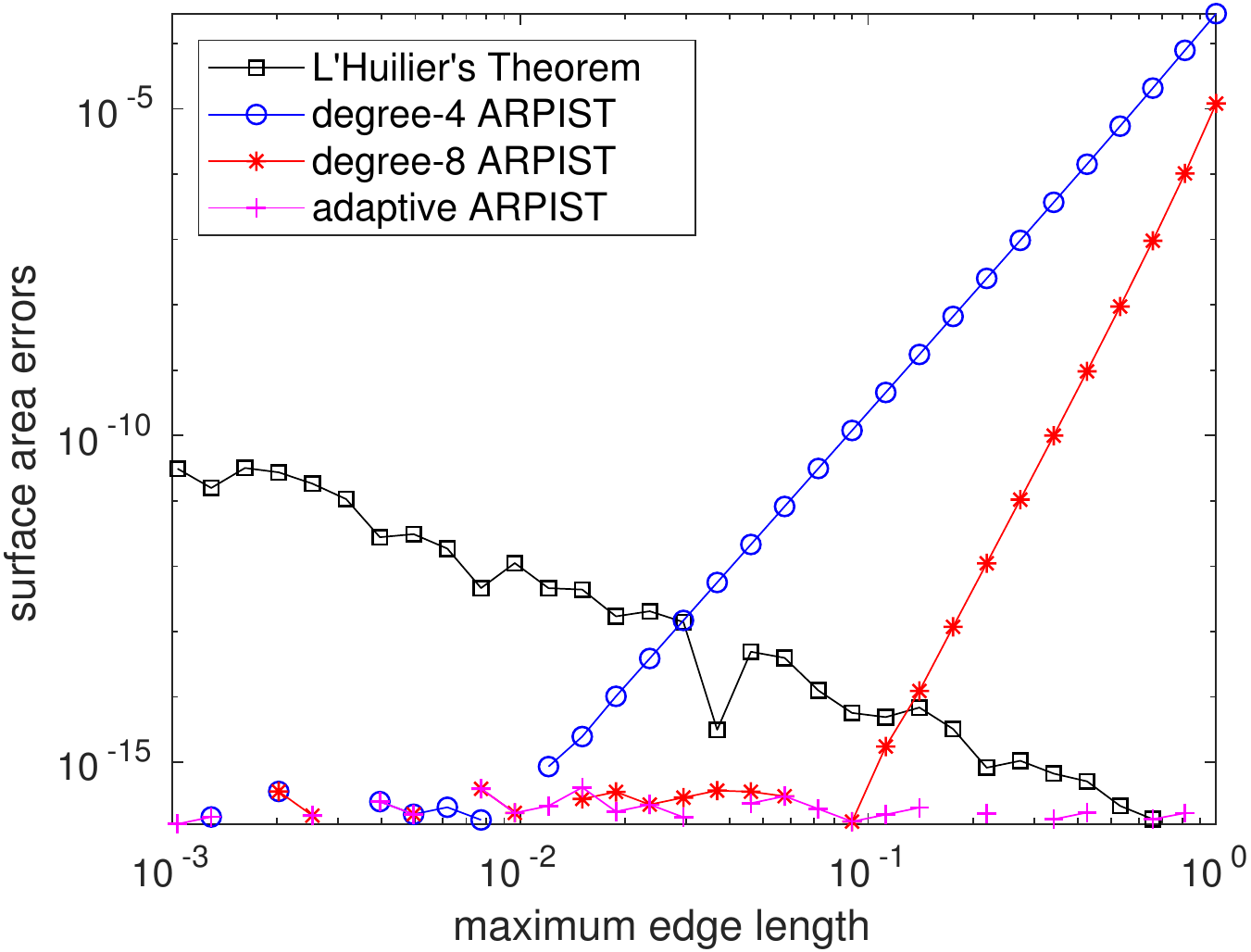}}
\caption{\label{fig:max-edge-length}Relative errors of computed areas of spherical
triangles using ARPIST versus L'Huilier's Theorem in double precision
for isosceles triangles with various minimum angles and edge lengths.
In (a), the errors of ARPIST do not change much for spherical triangles
with different minimum angles, while the errors of the L'Huilier's
Theorem increase as the minimum angles decrease. In (b), the errors
of degree-4 and degree-8 ARPIST decrease as the maximum edge length
decrease until they reach the machine precision, while the errors
of the L'Huilier's Theorem increase as the maximum edge lengths decrease.}
\end{figure}

In Figure~\ref{fig:max-edge-length}, it is also evident that the
relative errors in LT increased steadily as $\theta_{\min}$ or $h$
decreased. Despite its remarkably accuracy for $h\approx1$, the errors
of LT reached about $10^{-10}$ for poorly shaped large triangles
in Figure~\ref{fig:max-edge-length}(a) and $10^{-7}$ for poorly
shaped small triangles. The poor accuracy of LT for poor-shaped triangles
is due to the cancellation errors in $s-a_{i}$ in (\ref{eq:LHuilier}),
which are catastrophic when $\theta_{\min}$ is close to 0. The instability
of LT for small well-shaped triangles, on the other hand, is due to
the astronomical (absolute) condition number of the square-root operation,
i.e., 
\begin{equation}
\kappa(\sqrt{x})=\lim_{\epsilon\rightarrow0}\sup_{\epsilon}\frac{\vert\sqrt{x+\epsilon}-\sqrt{x}\vert}{\vert\epsilon\vert}=\frac{\text{d}\sqrt{x}}{\text{d}x}=\frac{1}{2\sqrt{x}},\label{eq:abs-cond-number}
\end{equation}
which tends to $\infty$ as $x$ approaches $0$. As $h$ approaches
$0$, the operand of the square-root operation in (\ref{eq:LHuilier})
tends to 0 because $s/2$, $(s-a_{i})/2$, and their tangents all
tend to $0$. Hence, this condition number in (\ref{eq:abs-cond-number})
drastically amplifies the rounding and cancellation errors, leading
to large errors for LT as seen in Figure~\ref{fig:max-edge-length}.
Hence, LT is unstable for fine meshes even with well-shaped triangles,
but it is particularly disastrous for those with poor-shaped small
triangles. To the best of our knowledge, adaptive ARPIST offers the
first viable alternative for general meshes to achieve (near) machine
precision for computing spherical-triangle areas, as long as $h$
and $\theta_{\min}$ are sufficiently large relative to $\epsilon_{\text{machine}}$.

\subsection{Integration of smooth analytic functions\label{subsec:Integration-of-smooth}}

To assess the accuracy and efficiency of ARPIST for integrating smooth
analytic function on spheres, we compare it with two techniques, namely
LSQST \citep{beckmann2014local} and SQRBF \citep{reeger2016numerical,reeger2015sqrbf}.
We chose these two techniques for comparison because they can be applied
to any given triangulation of a sphere and they were developed recently.

\subsubsection{Comparison with LSQST}

We first compare ARPIST with LSQST. The source code of LSQST is unavailable,
and its algorithm is very difficult to implement, so we applied the
adaptive ARPIST to solve a representative test problem as described
in Section~2.2 of \citep{beckmann2014local}. In particular, we integrate
the test function 
\begin{equation}
f^{(q,s)}(\boldsymbol{x})=\sum_{l=1}^{9}\alpha_{l}^{(q)}G_{s}\left(\left\langle \boldsymbol{x},\boldsymbol{\eta}_{l}^{(q)}\right\rangle _{2}\right),\label{eq:test-function}
\end{equation}
where the coefficients $\alpha_{l}^{(q)}$ and the centers $\boldsymbol{\eta}_{l}^{(q)}$
are randomly chosen, and
\[
G_{s}(t)=\frac{(1-s)^{3}}{(1-2st+s^{2})^{3/2}}
\]
is the Poisson kernel for some $s\in[0,1)$.\footnote{In \citep{beckmann2014local}, the authors used $h$ instead of $s$.
We use $s$ to avoid the confusion with edge length.} We chose the parameters
\[
s=0.8,0.9,0.95,0.97,0.975,0.98,0.985,0.99,0.995,
\]
which is a subset of those in \citep{beckmann2014local}. We excluded
the two cases $s=0.999$ and $s=0.9999$, since $G_{s}(t)$ tends
to a discontinuous function as $s$ approaches $1$ and robust resolution
of discontinuities is a separate topic in its own right (see, e.g.,
\citep{li2020wls}). As in \citep{beckmann2014local}, we compute
the average errors for 50 random pairs $(\alpha_{l}^{(q)},\boldsymbol{\eta}_{l}^{(q)})$,
\[
E(s)=\frac{1}{50}\sum_{q=1}^{50}\frac{\left|I_{\mathbb{S}^{2}}(f^{(q,s)})-Q_{\mathbb{S}^{2}}(f^{(q,s)})\right|}{\left|I_{\mathbb{S}^{2}}(f^{(q,s)})\right|},
\]
where 
\[
I_{\mathbb{S}^{2}}(f^{(q,s)})=\int_{\mathbb{S}^{2}}f^{(q,s)}\left(\left\langle \boldsymbol{x},\boldsymbol{\eta}_{l}^{(q)}\right\rangle _{2}\right)\,\text{d}\mu(x)=4\pi\frac{(1-s)^{2}}{1+s}
\]
and $Q_{\mathbb{S}^{2}}(f^{(q,s)})$ is the numerical integration
of $f^{(q,s)}$ using ARPIST or LSQST. As in \citep{beckmann2014local},
we generated a triangular mesh with 60 triangles using STRIPACK \citep{renka1997algorithm}
and then split one of the triangles into 256 small triangles as shown
in Figure~\ref{fig:average-quadrature-errors}(a). We applied the
adaptive ARPIST as described in Section~\ref{subsec:Area-computation-of}
on this test mesh, resulting in 786,432 quadrature points.

\begin{figure}
\centering{}\subfloat[Test mesh as used in \citep{beckmann2014local}.]{\includegraphics[width=0.42\textwidth]{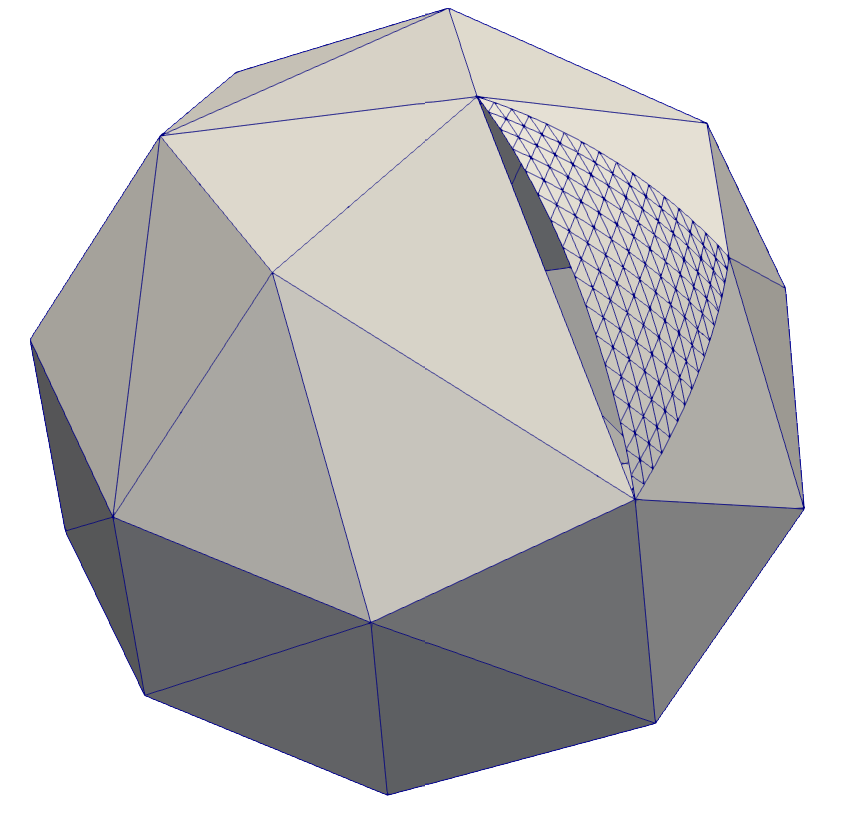}}
$\quad$\subfloat[Integration errors for different $s$.]{\includegraphics[width=0.55\textwidth]{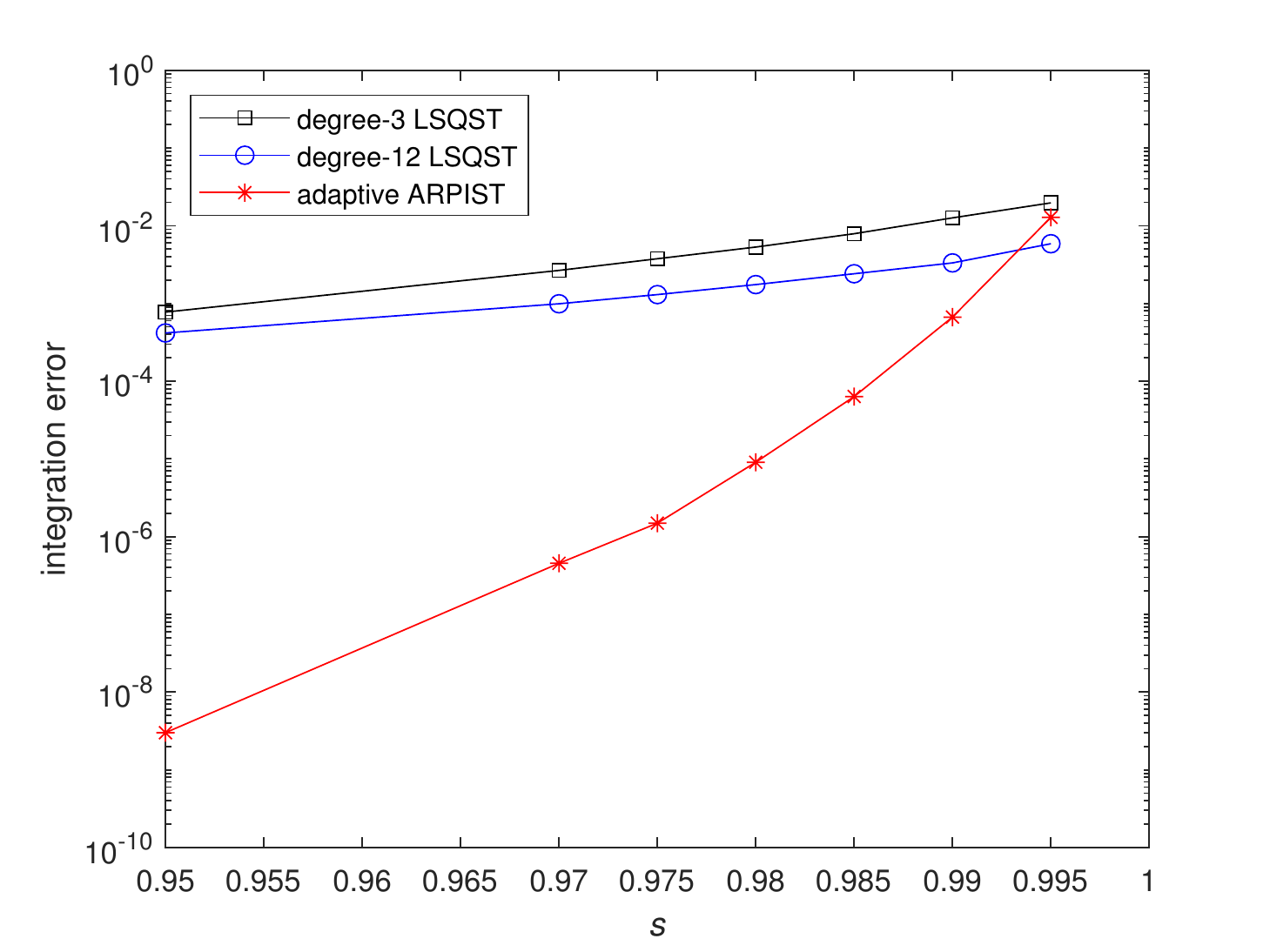}}\caption{\label{fig:average-quadrature-errors}Average quadrature errors $E(s)$
by integrating $f^{(q,s)}$ in (\ref{eq:test-function}) for different
$s$ on the test mesh.}
\end{figure}

As shown in Figure~\ref{fig:average-quadrature-errors}(b), the error
$E(s)$ for ARPIST ranged between $3\times10^{-9}$ and $1.28\times10^{-2}$
as $s$ increased. As points of reference, Figure~\ref{fig:average-quadrature-errors}(b)
reproduced the two representative results of degree-3 and degree-12
LSQST from Figure~2 of \citep{beckmann2014local} with 1,944,000
quadrature points. It can be seen that the errors from ARPIST are
about an order of magnitude smaller than LSQST for $s=0.99$. The
performance gap increased drastically as $s$ increased, and ARPIST
outperformed LSQST by five orders of magnitude for $s=0.95$. Remarkably,
ARPIST achieved this drastic improvement of accuracy with a much simpler
algorithm. The runtimes of LSQST were not reported in \citep{beckmann2014local}.
We estimate that ARPIST is at least an order of magnitude faster because
the computational costs of ARPIST and LSQST are linear and superlinear
in the number of quadrature points within each triangle, respectively.

\subsubsection{Comparison with SQRBF}

In this test, we compare ARPIST with the RBF-based spherical quadrature,
or SQRBF \citep{reeger2016numerical}. Since SQRBF has an open-source
MATLAB implementation \citep{reeger2015sqrbf}, we could conduct a
more in-depth comparison on a range of meshes. In particular, we used
STRIPACK \citep{renka1997algorithm} to generate a series of six Delaunay
triangulations of the unit sphere with $N=4^{2+i}$ nodes for $i=1,2,\dots,6$.
Figure~\ref{fig:test-mesh-SQRBF} shows three representative meshes.
\begin{figure}
\begin{raggedright}
\subfloat[$N=64$]{\raggedright{}\includegraphics[width=0.33\columnwidth]{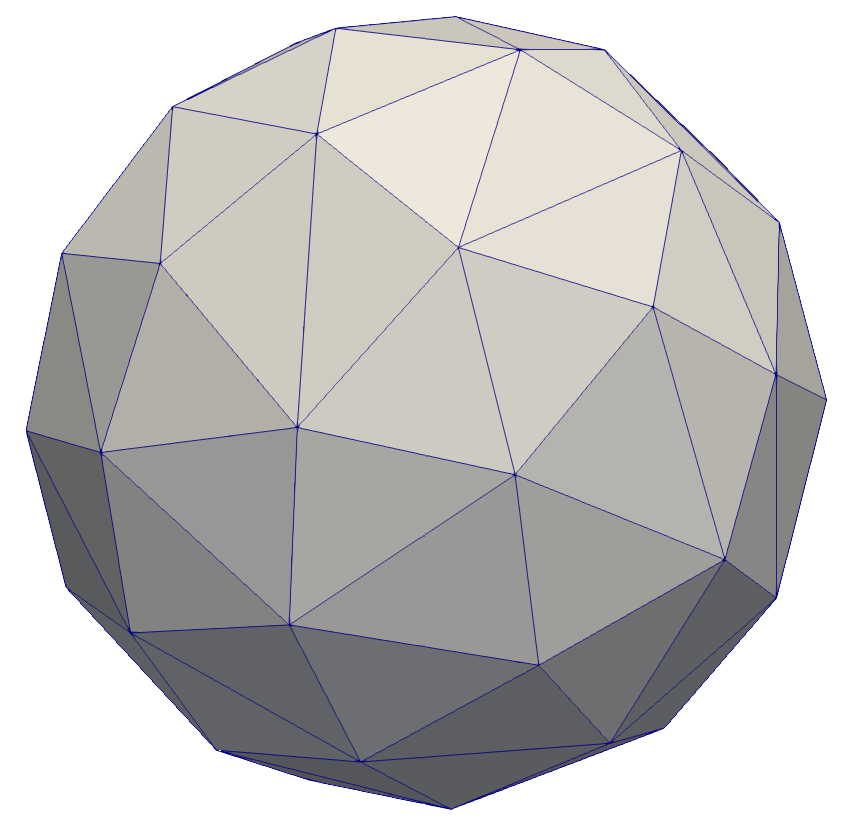}}\subfloat[$N=1024$]{\raggedright{}\includegraphics[width=0.3333\columnwidth]{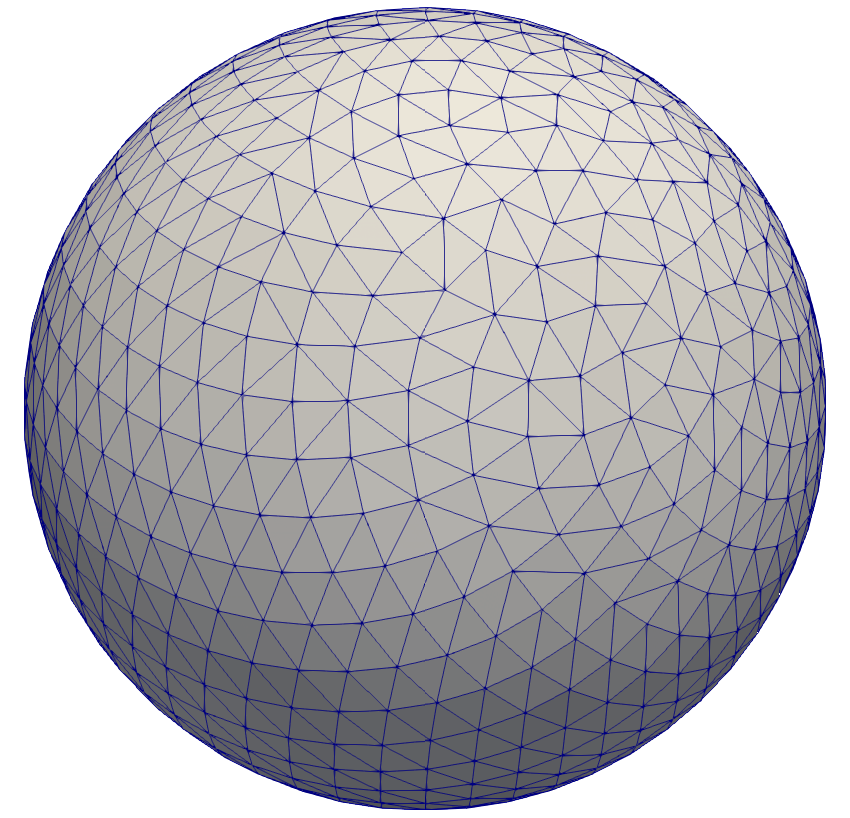}}\subfloat[$N=16384$]{\raggedright{}\includegraphics[width=0.3333\columnwidth]{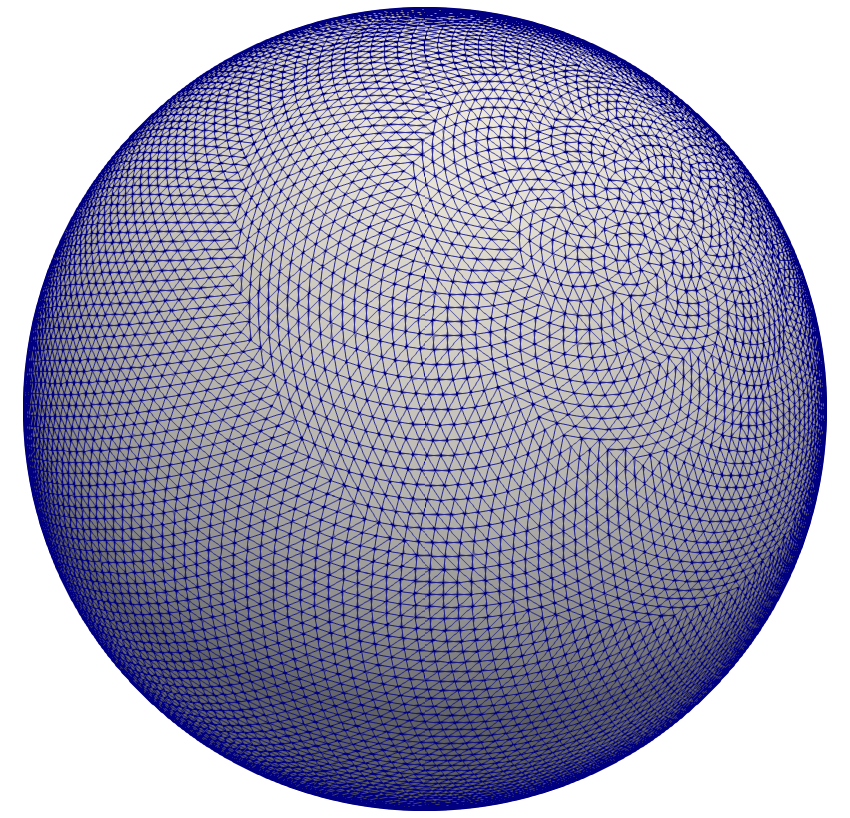}}
\par\end{raggedright}
\caption{\label{fig:test-mesh-SQRBF} Three representative test meshes in the
comparison between ARPIST and SQRBF.}
\end{figure}
 For ARPIST, we focused on degree-4 and degree-8 Gaussian quadrature
rules on each triangle, which have 6 and 16 quadrature points per
triangle, respectively. Hence, the total numbers of quadrature points
on the whole sphere are $12(4^{i}-2)$ and $32(4^{i}-2)$ for the
$i$th mesh.  For each technique, the integration of any smooth function
over the whole sphere is then a weighted sum of the function values
at the quadrature points. We compared the three techniques for several
test functions in \citep{reeger2016numerical}, and the results were
qualitatively the same. Hence, we present only the result for one
of the test functions,
\begin{align}
f_{1}(x,y,z) & =\frac{1}{9}(1+\tanh(9(z-x-y))),\label{eq:f1_test-function}
\end{align}
of which the exact integral over the sphere is $I_{\mathbb{S}^{2}}(f_{1})=4\pi/9$.

In Figure~\ref{fig:Comparison-ARPIST-SQRBF}, we compare ARPIST with
SQRBF in terms of accuracy and efficiency. Note that the different
methods have different numbers of quadrature points. Figures~\ref{fig:Comparison-ARPIST-SQRBF}(a)
and (b) show the relative integration errors with respect to the numbers
of quadrature points and the numbers of elements, respectively. It
can be seen that degree-8 ARPIST delivered better accuracy than SQRBF
while SQRBF was more accurate than degree-4 ARPIST. For completeness,
we also report the results for adaptive ARPIST, which achieved near
machine precision for all the meshes because its adaptive procedure
generated roughly the same numbers of quadrature points for coarser
meshes. In Figure~\ref{fig:Comparison-ARPIST-SQRBF}(c), we compare
the computational costs of the MATLAB implementations of ARPIST and
SQRBF. Both degree-4 and degree-8 ARPIST were about three orders of
magnitude faster than SQRBF. The adaptive ARPIST was also more efficient
than SQRBF, but it was less efficient than fixed-degree ARPIST for
coarser meshes. The cost of adaptive ARPIST can be further reduced
by enlarging its thresholds $h_{1}$ and $h_{2}$ for splitting the
triangles to reduce the number of quadrature points if lower-precision
solutions are needed. Hence, we conclude that ARPIST is much more
accurate, efficient, and robust than SQRBF. It is worth noting that
ARPIST is also much easier to implement, for example, in C++, to achieve
even greater performance. More importantly, ARPIST is more flexible
than SQRBF because it can be applied to individual triangles while
SQRBF only applies to a whole spherical triangulation.

\begin{figure}
\subfloat[~]{\includegraphics[width=0.33\textwidth]{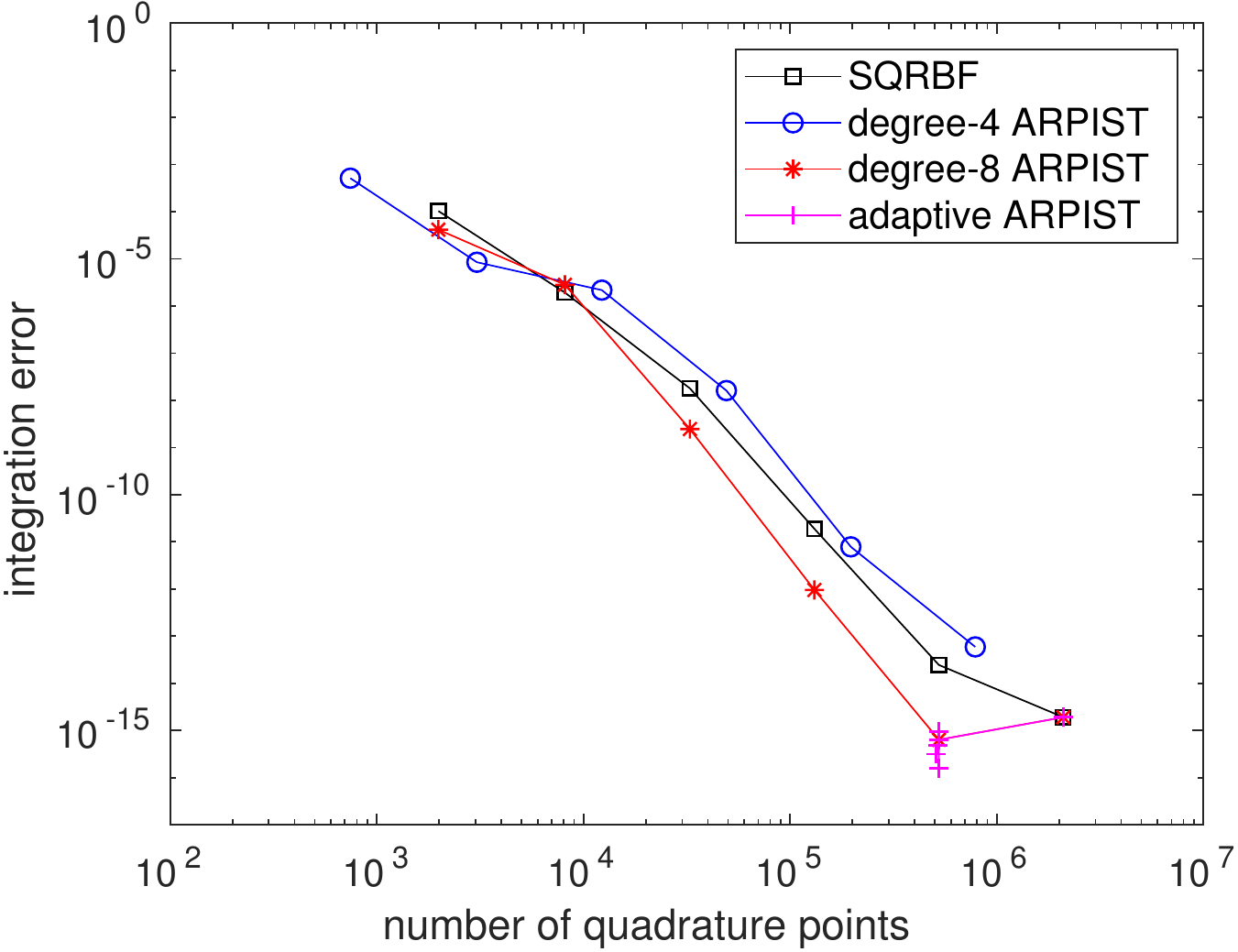}}~\subfloat[~]{\includegraphics[width=0.33\textwidth]{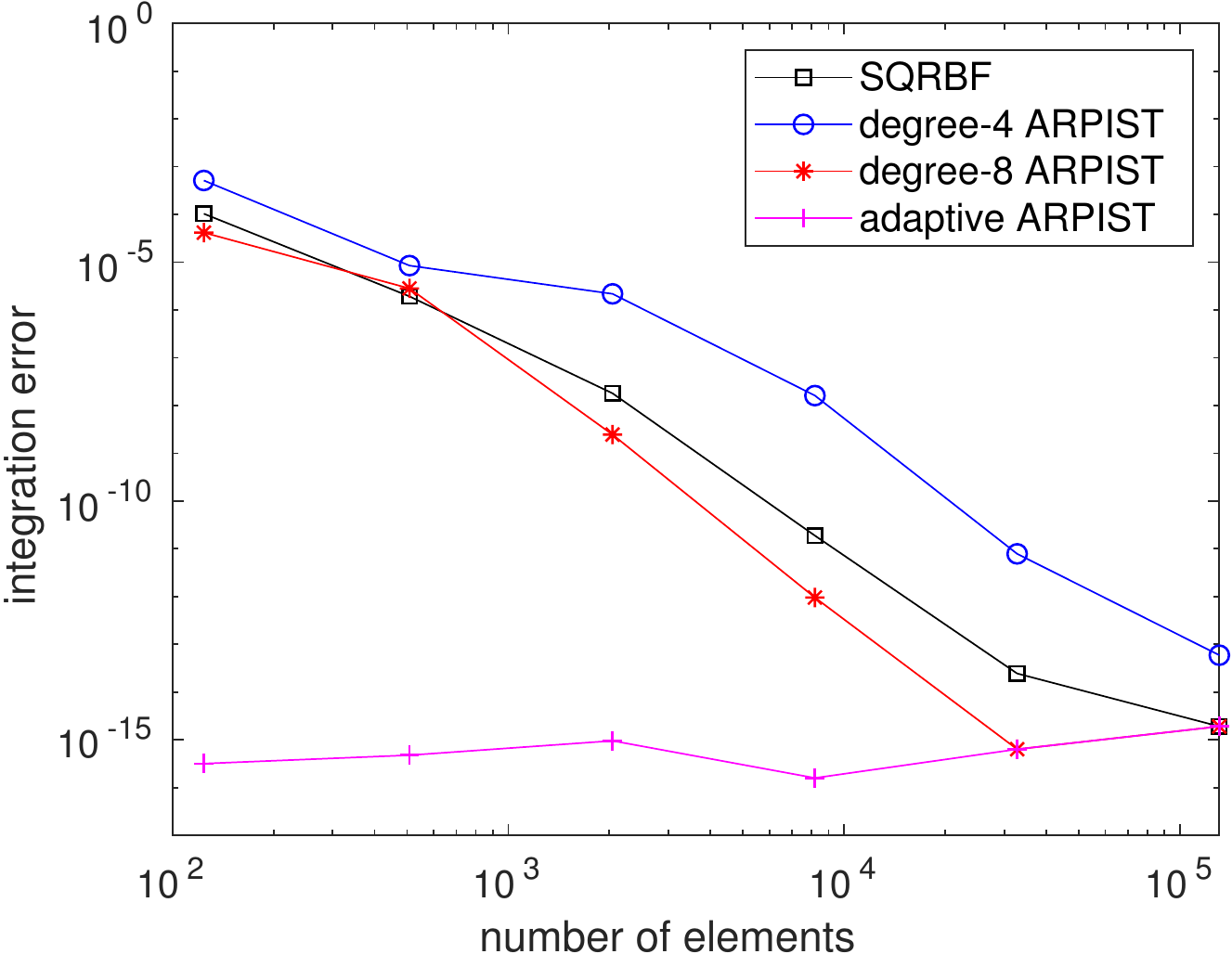}}~\subfloat[~]{\includegraphics[width=0.33\textwidth]{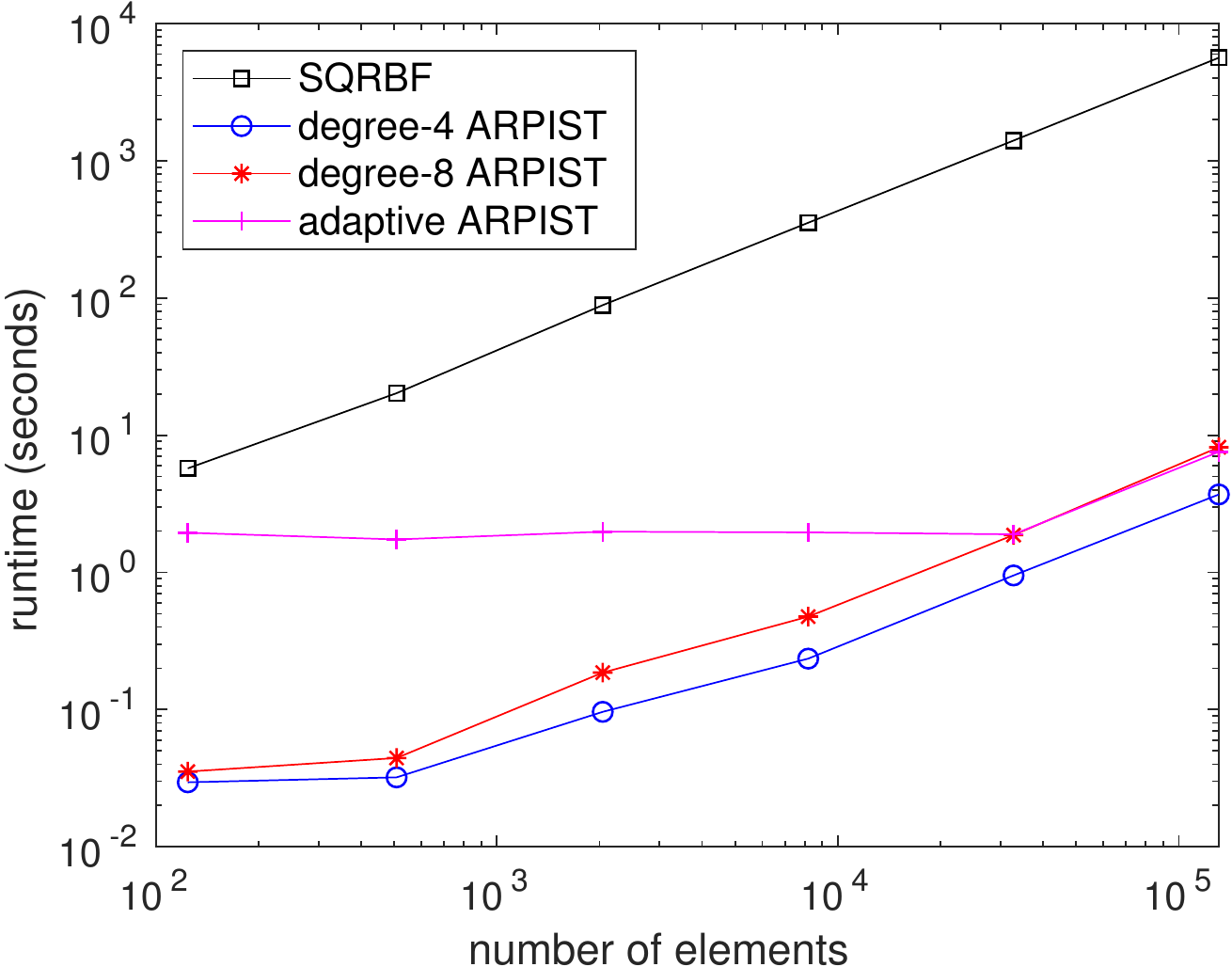}}

\caption{\label{fig:Comparison-ARPIST-SQRBF}Comparison of errors and runtimes
between ARPIST and SQRBF.}
\end{figure}

\subsection{Comparison for scattered data}

Our preceding examples consider analytical functions. In practice,
an analytic function may not be available and the function values
may be sampled at some given scattered data points, such as the nodes
of a given triangulation. The latter is the main assumption in LSQST
and SQRBF. Mathematically, it simply means that we must reconstruct
the values at the quadrature points from the scattered data values
using an interpolation or quasi-interpolation with comparable accuracy
to the quadrature rules. In the context of ARPIST, it can be achieved
by using a weighted-least-squares (WLS) reconstruction, similar to
that in \citep{RayWanJia12}. We omit the details of WLS and refer
readers to our previous works in \citep{RayWanJia12}, \citep{li2019compact},
or \citep{li2020wls} for details. For completeness, we briefly describe
how to couple WLS with ARPIST.

Suppose the function $f$ is sampled at discrete points $\{\boldsymbol{x}_{k}\}_{k=1}^{n}$,
and let $f_{k}=f(\boldsymbol{x}_{k})$. To compute the integration
over a given triangle $e_{i}$ on a sphere, we first use ARPIST to
generate the quadrature points $\{\boldsymbol{x}_{ij}\mid1\leq j\leq\ell_{i}\}$
and corresponding weights $\{w_{ij}\}$ in the triangle. Then, we
use WLS reconstruction to compute a sparse operator $\boldsymbol{A}_{i}\in\mathbb{R}^{\ell_{i}\times n}$
to interpolate the function values from the scattered points $\{\boldsymbol{x}_{k}\}_{k=1}^{n}$
to the quadrature points $\{\boldsymbol{x}_{ij}\}$. The spherical
integration operator over the triangle is 
\[
\boldsymbol{b}_{i}^{T}=\boldsymbol{w}_{i}^{T}\boldsymbol{A}_{i},
\]
where $\boldsymbol{w}_{i}\in\mathbb{R}^{\ell_{i}}$ is a column vector
composed of $w_{ij}$ and $\boldsymbol{b}_{i}\in\mathbb{R}^{n}$.
Given a column vector $\boldsymbol{f}=(f_{1},f_{2},...,f_{n})^{T}$,
the integral over $e_{i}$ is simply $\boldsymbol{b}_{i}^{T}\boldsymbol{f}$.
To obtain an integration operator over the complete triangulation,
one simply needs to add up $\boldsymbol{b}_{i}$ for all the triangles
$\{e_{i}\mid1\leq i\leq m$\}, i.e., $\boldsymbol{b}=\sum\boldsymbol{b}_{i}$.
Then, $\boldsymbol{b}^{T}\boldsymbol{f}$ is the total integral over
the whole sphere.

To assess the accuracy of ARPIST+WLS,\footnote{We are unable to compare with LSQST for this test due to the unavailability
of its source code.} we compare it with SQRBF for two test functions in \citep{reeger2016numerical},
namely $f_{1}$ in (\ref{eq:f1_test-function}) and 
\[
f_{2}(x,y,z)=\frac{1}{2}+\frac{\arctan(300(z-0.9999))}{\pi}.
\]
The exact integral of $f_{2}$ over the sphere is $I_{\mathbb{S}^{2}}(f_{2})\thickapprox0.014830900415995262852$
\citep{fuselier2014kernel}. We used STRIPACK \citep{renka1997algorithm}
to generate a series of Delaunay triangulations of the unit sphere
with $N=4^{4+i}$ nodes for $i=1,2,\dots,5$ and then sampled the
functions at the nodes of the triangulations. Since SQRBF can only
integrate over the whole sphere, we computed the operator $\boldsymbol{b}$
in ARPIST+WLS instead of $\boldsymbol{b}_{i}$ for the individual
triangles. Since ARPIST uses degree-4 and degree-8 quadrature rules,
we used degree-4, degree-6, and degree-8 WLS to match the accuracy
of the quadrature rules. As can be seen in Figure~\ref{fig:Comparison-ARPIST-SQRBF-1},
ARPIST with degree-8 WLS is comparable with SQRBF for $f_{1}$, and
it outperformed SQRBF on most of the meshes for $f_{2}$. However,
ARPIST with degree-4 and degree-6 WLS under-performed both SQRBF and
ARPIST+WLS-8 for finer meshes, because their interpolation errors
dominated the integration errors. It is worth noting that SQRBF uses
higher-degree polynomials than WLS-8, and it is designed for integrating
over the whole sphere only. 

\begin{figure}
\subfloat[\label{fig:aprist+wls f1}~]{\includegraphics[width=0.48\textwidth]{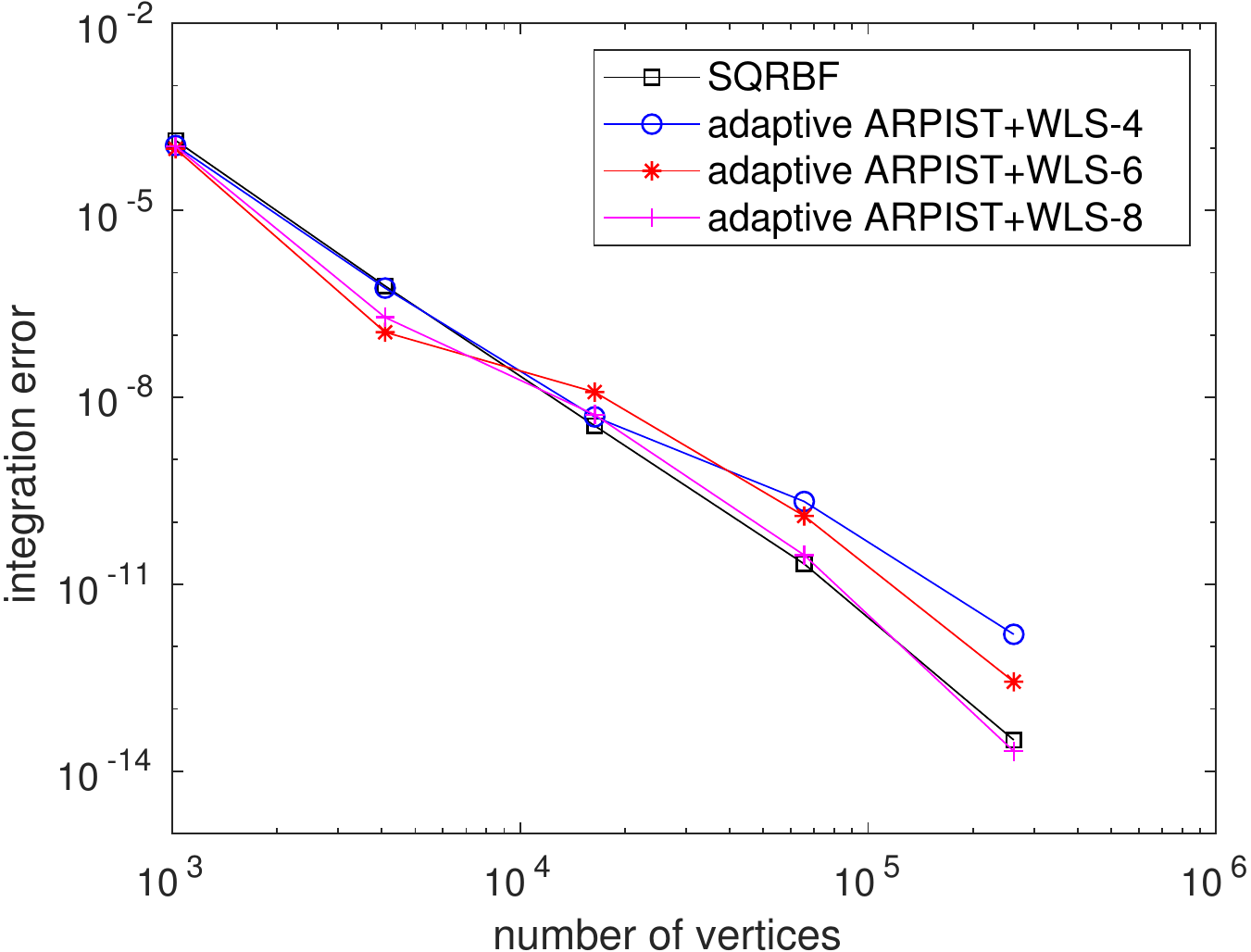}}~\subfloat[\label{fig:apritst+wls f2}~]{\includegraphics[width=0.48\textwidth]{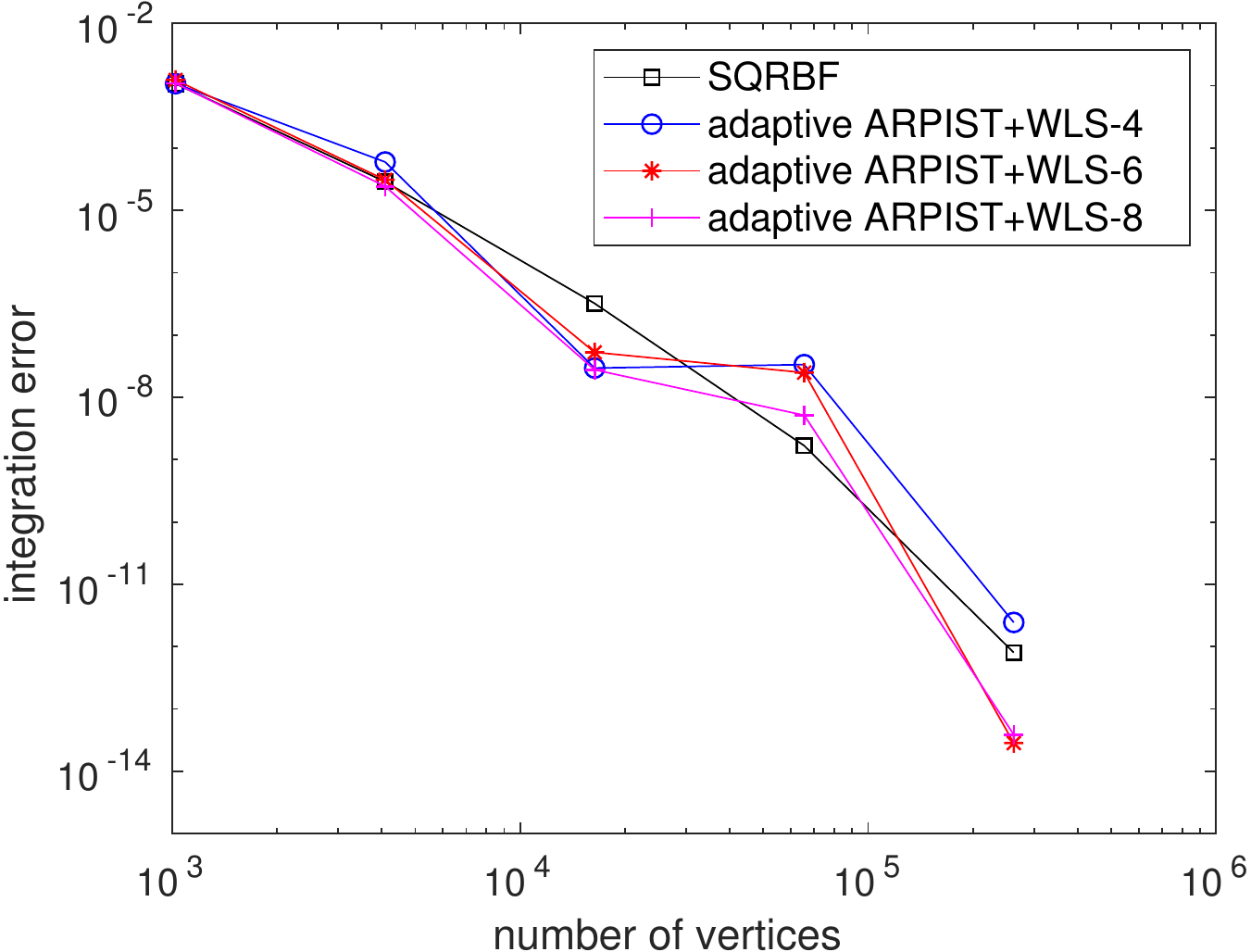}}

\caption{\label{fig:Comparison-ARPIST-SQRBF-1}Comparison of overall integration
errors between ARPIST+WLS and SQRBF.}
\end{figure}

\section{\label{sec:Conclusions}Conclusions}

In this work, we propose a new integration technique for spherical
triangles, called ARPIST. ARPIST utilizes a simple and effective transformation
from the spherical triangle to the linear triangle via radial projection
to achieve high accuracy and efficiency. More importantly, ARPIST
overcomes the potential instabilities in the Jacobian determinant
of the transformation to achieve provable accuracy and stability even
for poorly shaped triangles. Our experimental results verified that
ARPIST could reliably achieve (near) machine precision. We also showed
that ARPIST is orders of magnitude more accurate than the popular
technique of computing the area of spherical triangles based on L'Huilier's
Theorem. ARPIST is also more accurate and significantly more efficient
than other recently proposed techniques for integrating smooth functions
on spheres. When coupled with degree-8 WLS reconstructions, ARPIST
can integrate scattered data values with similar or better accuracy
compared to SQRBF. One limitation of this work is that it considered
only smooth functions. In addition, if the function has discontinuities,
then using a high-degree quadrature rule would generally lead to instabilities
due to the violation of the regularity assumptions of high-degree
quadrature rules. Some high-order limiters (such as WLS-ENO \citep{li2020wls})
are needed in this setting. We plan to address this issue in the future.

\section*{Acknowledgments}

This work was supported under the Scientific Discovery through Advanced
Computing (SciDAC) program in the US Department of Energy\textquoteright s
Office of Science, Office of Advanced Scientific Computing Research
through subcontract \#462974 with Los Alamos National Laboratory.
We thank Drs. Vijay S. Mahadevan and Paul Ullrich for helpful discussions
on spherical integration, which have motivated this work, and thank
Dr. Qiao Chen for his help in proofreading the paper. We thank the
anonymous reviewers for their helpful comments.

\bibliographystyle{elsarticle-num}
\bibliography{solution_transfer}

\appendix

\section{\label{sec:Error-analysis-of-ALUPP}Error analysis of LUPP with equilibration}

We prove the error bounds of the determinant using LUPP in Theorem~\ref{thm:forward-error-LUPP}
by adapting the standard backward error analysis in linear algebra.
\begin{proof}
 Without loss of generality, assume $\boldsymbol{a}_{2}$ and $\boldsymbol{a}_{3}$
are in counterclockwise order w.r.t. $\boldsymbol{a}_{1}$, so that
$\det(\boldsymbol{A})>0$. First, consider the equilibrated matrix
$\boldsymbol{B}=\boldsymbol{A}\boldsymbol{D}^{-1}$. Let $\tilde{\boldsymbol{P}}\boldsymbol{B}=\tilde{\boldsymbol{L}}\tilde{\boldsymbol{U}}$
be the LUPP with floating-point arithmetic. We claim that the computed
determinant is backward stable in the sense that there exists $\tilde{\boldsymbol{B}}=[\tilde{\boldsymbol{b}}_{1},\tilde{\boldsymbol{b}}_{2},\tilde{\boldsymbol{b}}_{3}]$
with $\Vert\tilde{\boldsymbol{b}}_{i}-\boldsymbol{b}_{i}\Vert=\mathcal{O}(\epsilon_{\text{machine}})\Vert\boldsymbol{b}_{i}\Vert=\mathcal{O}(\epsilon_{\text{machine}})$
for $i=1,2,3$ such that 
\[
\left|\det(\tilde{\boldsymbol{B}})\right|=\left|\prod_{i=1}^{3}\tilde{u}_{ii}\right|.
\]
This backward stability follows from the classical backward error
analysis of LUPP \citep[Theorem 22.2]{trefethen1997numerical}: There
exists $\tilde{\boldsymbol{B}}=[\tilde{\boldsymbol{b}}_{1},\tilde{\boldsymbol{b}}_{2},\tilde{\boldsymbol{b}}_{3}]$
with $\Vert\tilde{\boldsymbol{B}}-\boldsymbol{B}\Vert_{\infty}=\Vert\boldsymbol{B}\Vert_{\infty}\mathcal{O}(\epsilon_{\text{machine}})$
for $i=1,2,3$, such that $\tilde{\boldsymbol{P}}\tilde{\boldsymbol{B}}=\tilde{\boldsymbol{L}}\tilde{\boldsymbol{U}}$,
where $\tilde{\boldsymbol{L}}$ is unit lower triangular (i.e., with
ones along its diagonal), and $\tilde{\boldsymbol{P}}$ is another
permutation matrix. Hence, $\left|\det(\tilde{\boldsymbol{B}})\right|=\left|\det(\tilde{\boldsymbol{U}})\right|=\left|\prod_{i=1}^{3}\tilde{u}_{ii}\right|$.
Furthermore, under the assumption of $\Vert\boldsymbol{b}_{i}\Vert=1$,
$1\leq\Vert\boldsymbol{B}\Vert\leq\sqrt{3}$, so $\Vert\tilde{\boldsymbol{b}}_{i}-\boldsymbol{b}_{i}\Vert\leq\Vert\tilde{\boldsymbol{B}}-\boldsymbol{B}\Vert=\mathcal{O}(\epsilon_{\text{machine}})$.

Second, the absolute condition number of $\text{\ensuremath{\det}}(\boldsymbol{B})$
w.r.t. perturbations in $b_{ij}$ is $\left|\frac{\partial\text{\ensuremath{\det}}(\boldsymbol{B})}{\partial b_{ij}}\right|=\left|\det(\boldsymbol{B}_{ij})\right|$,
where $\boldsymbol{B}_{ij}$ denotes the $2$-by-$2$ matrix obtained
by removing the $i$th row and $j$th column in $\boldsymbol{B}$.
The absolute condition number of $\text{\ensuremath{\det}}(\boldsymbol{B})$
in $\infty$-norm is $\kappa_{\text{det}}=\left\Vert \left[\frac{\partial\text{\ensuremath{\det}}(\boldsymbol{B})}{\partial b_{ij}}\right]_{ij}\right\Vert _{\infty}$,
and 
\begin{equation}
\left|\det(\tilde{\boldsymbol{B}})-\det(\boldsymbol{B})\right|\leq\kappa_{\text{det}}\Vert\tilde{\boldsymbol{B}}-\boldsymbol{B}\Vert_{\infty}=\kappa_{\text{det}}\Vert\boldsymbol{B}\Vert_{\infty}\mathcal{O}(\epsilon_{\text{machine}}).\label{eq:error-determinant}
\end{equation}
Given that $\boldsymbol{B}$ is equilibrated, $\kappa_{\text{det}}\leq6$
and $\Vert\boldsymbol{B}\Vert_{\infty}\leq3$, so $\left|\left|\det(\tilde{\boldsymbol{B}})\right|-\left|\det(\boldsymbol{B})\right|\right|\leq\left|\det(\tilde{\boldsymbol{B}})-\det(\boldsymbol{B})\right|=\mathcal{O}(\text{\ensuremath{\epsilon}}_{\text{machine}})$. 

Finally, let $\tilde{d}_{i}$ be $\Vert\boldsymbol{a}_{i}\Vert$ computed
in floating-point arithmetic.
\[
\det(\tilde{\boldsymbol{D}})=\prod_{i}\tilde{d}_{i}=(1+\mathcal{O}(\epsilon_{\text{machine}}))\text{\ensuremath{\det}(\ensuremath{\boldsymbol{D}})}.
\]
Therefore,
\begin{align*}
 & \left|\left|\prod_{i=1}^{3}d_{i}\tilde{u}_{ii}\right|-\left|\prod_{i=1}^{3}d_{i}u_{ii}\right|\right|\\
\leq & \left|\prod_{i=1}^{3}d_{i}\tilde{u}_{ii}-\prod_{i=1}^{3}d_{i}u_{ii}\right|\\
= & \left|\det(\tilde{\boldsymbol{D}})\det(\tilde{\boldsymbol{B}})-\det(\boldsymbol{D})\det(\boldsymbol{B})\right|\\
= & \left|\det(\tilde{\boldsymbol{D}})\det(\tilde{\boldsymbol{B}})-\det(\boldsymbol{D})\det(\tilde{\boldsymbol{B}})+\det(\boldsymbol{D})\det(\tilde{\boldsymbol{B}})-\det(\boldsymbol{D})\det(\boldsymbol{B})\right|\\
\leq & \det(\tilde{\boldsymbol{B}})\left|\det(\tilde{\boldsymbol{D}})-\det(\boldsymbol{D})\right|+\det(\boldsymbol{D})\left|\det(\tilde{\boldsymbol{B}})-\det(\boldsymbol{B})\right|\\
= & \det(\boldsymbol{D})\mathcal{O}(\epsilon_{\text{machine}}).
\end{align*}
\end{proof}
In the proof, column equilibration played an important role. Without
equilibration, it would be more difficult to bound $\kappa_{\det}\Vert\boldsymbol{B}\Vert_{\infty}$
in (\ref{eq:error-determinant}). Nevertheless, LUPP without equilibration
turned out to perform well in practice.

\section{\label{sec:Anchored-triple-product}Error analysis of anchored triple
product}

We now generalize the error analysis in \ref{sec:Error-analysis-of-ALUPP}
to prove Theorem~\ref{thm:forward-error-TP} for ATP. The algorithm
for ATP is simpler, but its analysis needs to deviate from the standard
backward error analysis. To this end, we define $\sigma\coloneqq\left|\boldsymbol{a}\cdot\left(\boldsymbol{b}\times\boldsymbol{c}\right)\right|/\Vert\boldsymbol{a}\Vert\Vert\boldsymbol{b}\Vert\Vert\boldsymbol{c}\Vert$,
so that $1/\sigma$ will play the role of the \emph{relative condition
number} under the assumption that $\sigma\gg\epsilon_{\text{machine}}$.
This simplification is necessary because the triple product is a nonlinear
(quadratic) operation, so unlike in linear algebra, one cannot give
a simple closed form for the condition number when $\sigma\approx\epsilon_{\text{machine}}$.
\begin{proof}
We first show that the relative error in the triple product $\boldsymbol{a}\cdot(\boldsymbol{b}\times\boldsymbol{c})$
is approximately bounded by $\mathcal{O}(\epsilon_{\text{machine}})/\sigma$
when $\sigma\gg\epsilon_{\text{machine}}$. Without loss of generality,
assume that $a_{i}$, $b_{i}$, and $c_{i}$ are all floating-point
numbers. Then,
\begin{align*}
\boldsymbol{b}\otimes\boldsymbol{c} & =\begin{bmatrix}\left(b_{2}c_{3}(1+\epsilon_{1})-b_{3}c_{2}(1+\epsilon_{2})\right)(1+\epsilon_{7})\\
\left(b_{3}c_{1}(1+\epsilon_{3})-b_{1}c_{3}(1+\epsilon_{4})\right)(1+\epsilon_{8})\\
\left(b_{1}c_{2}(1+\epsilon_{5})-b_{2}c_{1}(1+\epsilon_{6})\right)(1+\epsilon_{9})
\end{bmatrix}\\
 & =\begin{bmatrix}b_{2}c_{3}(1+\epsilon_{2}')-b_{3}c_{2}(1+\epsilon_{6}')\\
b_{3}c_{1}(1+\epsilon_{3}')-b_{1}c_{3}(1+\epsilon_{4}')\\
b_{1}c_{2}(1+\epsilon_{1}')-b_{2}c_{1}(1+\epsilon_{5}')
\end{bmatrix},
\end{align*}
where $\left|\epsilon_{i}\right|\leq\epsilon_{\text{machine}}$ and
$\left|\epsilon_{i}'\right|\leq2\epsilon_{\text{machine}}+\mathcal{O}(\epsilon_{\text{machine}}^{2})$.
Let $\tilde{\boldsymbol{v}}$ denote $\boldsymbol{b}\otimes\boldsymbol{c}$.
Then, 
\begin{align*}
\boldsymbol{a}\odot(\boldsymbol{b}\otimes\boldsymbol{c}) & =\left(\left(a_{1}(1+\epsilon_{10})\tilde{v}_{1}+a_{2}(1+\epsilon_{11})\tilde{v}_{2}\right)(1+\epsilon_{12})+a_{3}(1+\epsilon_{13})\tilde{v}_{3}\right)(1+\epsilon_{14})\\
 & =\left(a_{1}(1+\epsilon_{a1})\right)\tilde{v}_{1}+\left(a_{2}(1+\epsilon_{a2})\right)\tilde{v}_{2}+\left(a_{3}(1+\epsilon_{a3})\right)\tilde{v}_{3}
\end{align*}
where $\left|\epsilon_{ai}\right|\leq3\epsilon_{\text{machine}}+\mathcal{O}(\epsilon_{\text{machine}}^{2})$
for $1\leq i\leq2$ and $\left|\epsilon_{a3}\right|\leq2\epsilon_{\text{machine}}+\mathcal{O}(\epsilon_{\text{machine}}^{2})$.
Let $\tilde{\boldsymbol{a}}=\left[a_{i}(1+\epsilon_{ai})\right]_{i}$
and $\tilde{\boldsymbol{b}}=\left[b_{i}(1+\epsilon_{i}')\right]_{i}$
for $1\leq i\leq3$, and then 
\begin{align}
\boldsymbol{a}\odot(\boldsymbol{b}\otimes\boldsymbol{c})-\tilde{\boldsymbol{a}}\cdot(\tilde{\boldsymbol{b}}\times\boldsymbol{c}) & =\tilde{\boldsymbol{a}}\cdot(\tilde{\boldsymbol{b}}\times\boldsymbol{c})-\tilde{\boldsymbol{a}}\cdot\begin{bmatrix}\tilde{b}_{2}c_{3}-\tilde{b}_{3}c_{2}\frac{1+\epsilon_{6}'}{1+\epsilon_{3}'}\\
\tilde{b}_{3}c_{1}-\tilde{b}_{1}c_{3}\frac{1+\epsilon_{4}'}{1+\epsilon_{1}'}\\
\tilde{b}_{1}c_{2}-\tilde{b}_{2}c_{1}\frac{1+\epsilon_{5}'}{1+\epsilon_{2}'}
\end{bmatrix}\nonumber \\
 & =\tilde{a}_{1}\tilde{b}_{3}c_{2}\epsilon_{1}''+\tilde{a}_{2}\tilde{b}_{1}c_{3}\epsilon_{2}''+\tilde{a}_{3}\tilde{b}_{2}c_{1}\epsilon_{3}'',\label{eq:remainder}
\end{align}
where $\left|\epsilon_{i}''\right|\leq4\epsilon_{\text{machine}}+\mathcal{O}(\epsilon_{\text{machine}}^{2})$.
Assuming $\sigma\gg\epsilon_{\text{machine}}$, it is easy to show
that 
\[
\left\Vert \tilde{\boldsymbol{a}}\cdot\left(\tilde{\boldsymbol{b}}\times\boldsymbol{c}\right)\right\Vert \geq\tilde{\sigma}\Vert\tilde{\boldsymbol{a}}\Vert\Vert\tilde{\boldsymbol{b}}\Vert\Vert\boldsymbol{c}\Vert,
\]
where $\tilde{\sigma}=\sigma(1+\mathcal{O}(\text{\ensuremath{\epsilon}}_{\text{machine}}))$,
and
\begin{align*}
\left|\tilde{a}_{1}\tilde{b}_{3}c_{2}\epsilon_{1}''+\tilde{a}_{2}\tilde{b}_{1}c_{3}\epsilon_{2}''+\tilde{a}_{3}\tilde{b}_{2}c_{1}\epsilon_{3}''\right| & \leq\Vert\tilde{\boldsymbol{a}}\Vert\Vert\tilde{\boldsymbol{b}}\Vert\Vert\boldsymbol{c}\Vert\mathcal{O}(\text{\ensuremath{\epsilon}}_{\text{machine}})\\
 & \leq\frac{1}{\tilde{\sigma}}\left\Vert \tilde{\boldsymbol{a}}\cdot\left(\tilde{\boldsymbol{b}}\times\boldsymbol{c}\right)\right\Vert \mathcal{O}(\text{\ensuremath{\epsilon}}_{\text{machine}})\\
 & =\frac{1}{\sigma}\left\Vert \tilde{\boldsymbol{a}}\cdot\left(\tilde{\boldsymbol{b}}\times\boldsymbol{c}\right)\right\Vert \mathcal{O}(\text{\ensuremath{\epsilon}}_{\text{machine}}).
\end{align*}
If the input numbers were not yet floating-point numbers, we only
need to increase the constant factors for the $\epsilon$, and the
asymptotic argument still holds.

Second, substituting $\boldsymbol{a}_{1}$, $\boldsymbol{a}_{2}$,
and $\boldsymbol{a}_{3}$ for $\boldsymbol{a}$, $\boldsymbol{b}$,
and $\boldsymbol{c}$ in the above, and following a similar argument
as for ALUPPE in \ref{sec:Error-analysis-of-ALUPP}, we obtain 
\begin{align*}
\left|\boldsymbol{a}_{1}\cdot(\boldsymbol{a}_{2}\times\boldsymbol{a}_{3})-\tilde{\boldsymbol{a}}_{1}\cdot(\tilde{\boldsymbol{a}}_{2}\times\boldsymbol{a}_{3})\right| & =\prod_{i}\Vert\boldsymbol{a}_{i}\Vert\mathcal{O}(\epsilon_{\text{machine}})\\
 & \leq\frac{1}{\sigma}\det(\boldsymbol{A})\mathcal{O}(\epsilon_{\text{machine}})
\end{align*}
under the assumptions that $\boldsymbol{a}_{2}$ and $\boldsymbol{a}_{3}$
are in counterclockwise order with respect to $\boldsymbol{a}_{1}$
and that $\left|\boldsymbol{a}_{1}\cdot(\boldsymbol{a}_{2}\times\boldsymbol{a}_{3})\right|\geq\sigma\prod_{i}\Vert\boldsymbol{a}_{i}\Vert$
for constant $\sigma\gg\epsilon_{\text{machine}}$. Similarly,
\begin{align*}
\left|\boldsymbol{a}_{1}\odot(\boldsymbol{a}_{2}\otimes\boldsymbol{a}_{3})-\tilde{\boldsymbol{a}}_{1}\cdot(\tilde{\boldsymbol{a}}_{2}\times\boldsymbol{a}_{3})\right| & \leq\frac{1}{\sigma}\det(\boldsymbol{A})\mathcal{O}(\text{\ensuremath{\epsilon}}_{\text{machine}}).
\end{align*}
Hence, 
\begin{align*}
 & \left|\boldsymbol{a}_{1}\cdot(\boldsymbol{a}_{2}\times\boldsymbol{a}_{3})-\boldsymbol{a}_{1}\odot(\boldsymbol{a}_{2}\otimes\boldsymbol{a}_{3})\right|\\
\leq & \left|\boldsymbol{a}\cdot(\boldsymbol{b}\times\boldsymbol{c})-\tilde{\boldsymbol{a}}\cdot(\tilde{\boldsymbol{b}}\times\boldsymbol{c})\right|+\left|\boldsymbol{a}\odot(\boldsymbol{b}\otimes\boldsymbol{c})-\tilde{\boldsymbol{a}}\cdot(\tilde{\boldsymbol{b}}\times\boldsymbol{c})\right|\\
\le & \frac{1}{\sigma}\det(\boldsymbol{A})\mathcal{O}(\epsilon_{\text{machine}}).
\end{align*}
\end{proof}

\end{document}